\documentclass[11pt,a4j,twoside]{article}

\usepackage{amsmath,amssymb}
\usepackage{amsthm}
\usepackage[abbrev]{amsrefs}
\usepackage{latexsym}
\usepackage{graphicx}

\usepackage[english]{babel}
\usepackage{fancyhdr}
\usepackage{a4wide}

\usepackage{bm}

\allowdisplaybreaks[1]

\numberwithin{equation}{section}

\newtheorem{thm}{Theorem}[section]
\newtheorem{prop}[thm]{Proposition}
\newtheorem{lem}[thm]{Lemma}

\newtheorem{dfn}[thm]{Definition}
\theoremstyle{definition} 
\newtheorem{ex}[thm]{Example}
\newtheorem{rem}[thm]{Remark}
\newtheorem{as}[thm]{Assumption}

\newcommand\ND{\newcommand}

\ND\lref[1]{Lemma~\ref{#1}}
\ND\tref[1]{Theorem~\ref{#1}}
\ND\pref[1]{Proposition~\ref{#1}}
\ND\sref[1]{Section~\ref{#1}}
\ND\ssref[1]{Subsection~\ref{#1}}
\ND\aref[1]{Appendix~\ref{#1}}
\ND\rref[1]{Remark~\ref{#1}}
\ND\cref[1]{Corollary~\ref{#1}}
\ND\eref[1]{Example~\ref{#1}}
\ND\fref[1]{Fig.\ {#1} }
\ND\lsref[1]{Lemmas~\ref{#1}}
\ND\tsref[1]{Theorems~\ref{#1}}
\ND\dref[1]{Definition~\ref{#1}}
\ND\psref[1]{Propositions~\ref{#1}}
\ND\rsref[1]{Remarks~\ref{#1}}
\ND\sssref[1]{Subsections~\ref{#1}}
\ND\esref[1]{Examples~\ref{#1}}
\ND\asref[1]{Assumption~\ref{#1}}

\newcommand{\ep}{\varepsilon}
\newcommand{\what}{\widehat}
\newcommand{\h}{\quad}
\newcommand{\dis}{\displaystyle}
\newcommand{\cl}{c\`adl\`ag\ }
\newcommand{\Prob}{\mathbb{P}}

\title{Probabilistic characterization of weakly harmonic maps with respect to non-local Dirichlet forms}
\date{}
\author{Fumiya Okazaki}
\begin{document}
\maketitle
\footnote{Mathematical Institute, Graduate School of Science, Tohoku University, Sendai, Japan}
\footnote{Email address: fumiya.okazaki.q4@dc.tohoku.ac.jp}
\renewcommand{\thepage}{\arabic{page}}
\begin{abstract}
We characterize weakly harmonic maps with respect to non-local Dirichlet forms by Markov processes and martingales. In particular, we can obtain discontinuous martingales on Riemannian manifolds from the image of symmetric stable processes under fractional harmonic maps in a weak sense. Based on this characterization, we also consider the continuity of weakly harmonic maps along the paths of Markov processes and describe the condition for the continuity of harmonic maps by quadratic variations of martingales in some situations containing cases of energy minimizing maps.
\end{abstract}

\flushleft{ {\bf MSC2020 Subject Classifications:} 60J46, 58E20}

\section{Introduction}
Harmonic maps are critical points of the Dirichlet energy defined on the space of maps between two manifolds. There are a lot of studies about harmonic maps from both analytic and geometric points of view. The regularity of harmonic maps is one of the most important problems. In general it is known that harmonic maps have singular points by the non-linearity of the Euler-Lagrange equation obtained from the variational problem for the Dirichlet energy. The theory of partial regularity for harmonic maps was obtained in \cite{SU82, SU84} for energy minimizing maps, and in \cite{Evans, Bethuel} for stationary harmonic maps.

On the other hand, harmonic maps are related to the theory of probability and they also have been studied through stochastic processes. In fact, we can easily show by It\^o's formula that smooth harmonic maps map Brownian motions on domain manifolds to martingales on target manifolds. In \cite{Pic2001}, it was shown that this probabilistic characterization of harmonic maps is valid even for weakly harmonic maps with respect to strongly local Dirichlet forms. Regularity of harmonic maps has been also considered form the view point of the theory of probability. The regularity of harmonic maps with values in spaces with a kind of convex geometry has been shown in \cite{ALT99, Pic2000}.

Recently, fractional harmonic maps, which are critical points of the fractional Dirichlet energy, are also studied. The study of fractional harmonic maps is initiated in \cite{Lio, Lio2} and the regularity theory for $\frac{1}{2}$-harmonic lines was obtained in those articles. The partial regularity for $\frac{1}{2}$-harmonic maps has been shown in \cite{MSire}. Recently, for $\alpha \in (0,2)$, the partial regularity for $\frac{\alpha}{2}$-harmonic maps with values in spheres has been obtained in \cite{MPS}.

In this paper, we first obtain a probabilistic characterization of harmonic maps with respect to regular Dirichlet forms which are not necessarily strongly local. This result is the extension of that of \cite{Pic2001}. Fractional harmonic mas are one of the most typical examples of such harmonic maps. To study harmonic maps for general regular Dirichlet forms by stochastic processes, we will apply the theory of martingales with jumps on a submanifold. Discontinuous martingales on manifolds are first introduced in \cite{Pic1991} and further studied in \cite{Pic1994}. Recently, some properties concerning the convergence of discontinuous martingales have been studied in the author's previous paper \cite{Oka23} for the purpose of applications to harmonic maps. 

To state the characterization of harmonic maps by martingales on submanifold in a general situation, we briefly specify the setting. Let $E$ be a locally compact and separable metric space, $m$ a positive Radon measure with full support on $E$, and $(\mathcal{E},\mathcal{F})$ a regular Dirichlet form on $L^2(E;m)$. For a special standard core $\mathcal{C}$ and an open set $D\subset E$, we denote by $\mathcal{C}_D$ the family of functions in $\mathcal{C}$ whose supports are included in $D$. Let $M$ be a Riemannian submanifold of a higher dimensional Euclidean space $\mathbb{R}^d$. For a map $u=(u^1,\dots,u^d)\colon E\to M \subset \mathbb{R}^d$ such that $u^i\in \mathcal{F}^D_{loc}$ satisfies some conditions, we can consider the following equation:
\begin{gather}
\sum_{i=1}^d\mathcal{E}(u^i,\psi^i)=0,\ \text{for all}\ \psi \in \mathcal{C}_D(\mathbb{R}^d)\ \text{with}\ \psi (z)\in T_{ u(z)}M\ \text{for}\ m\text{-a.e.}\ z\in E,\label{EulerLagrange}
\end{gather}
where
\[
\mathcal{C}_D(\mathbb{R}^d)=\{ \psi=(\psi^1,\dots,\psi^d)\colon E \to \mathbb{R}^d \mid \psi^i\in \mathcal{C}_D\ \text{for each}\ i=1,\dots,d \}.
\]
We can regard \eqref{EulerLagrange} as the Euler-Lagrange equation with respect to the energy functional $\mathcal{E}$. A map $u \colon E\to M$ satisfying \eqref{EulerLagrange} is called a weakly harmonic map on $D$. On the other hand, a regular Dirichlet form determines a symmetric Hunt process on $E$. We define a quasi-harmonic map as a map which maps the symmetric Hunt process to a martingale on $M$. The equivalence of weakly harmonic maps and quasi-harmonic maps was shown in \cite{Pic2001} in the case where $(\mathcal{E},\mathcal{F})$ is strongly local transient regular Dirichlet form and $u$ is in the reflected Dirichlet space. On the other hand, harmonic functions in $\mathcal{F}_{loc}^D$ for general regular Dirichlet form were characterized by stochastic processes in \cite{Chen}. (In this paper, we mainly refer to Chapter 6 of \cite{ChenFuku}.) By combining \cite{Pic2001, Chen}, and the theory of stochastic integrals along CAF's of zero energy developed in \cite{Nakao, CFK, Kuwae, Walsh}, we obtain the non-local version of the result of \cite{Pic2001}. To simplify the statement, we only state the case that a target manifold $M$ is compact in \tref{harmonicmartingale} below, but we will show this kind of result under more general situations in \tsref{harmonicmartingale1} and \ref{harmonicmartingale2} and those results in \sref{harmonicmap} include Proposition 4 of \cite{Pic2001}.
\begin{thm}\label{harmonicmartingale}
Let $M$ be a compact Riemannian submanifold of $\mathbb{R}^d$. We assume that a Borel measurable map $u:E\to M$ is in $\mathcal{F}_{loc}^D(M)$ and quasi-continuous on $D$. Then $u$ is weakly harmonic on $D$ if and only if $u$ is quasi-harmonic on $D$.
\end{thm}
As a particular case of \tref{harmonicmartingale}, we can obtain a discontinuous martingale on a submanifold by substituting an $\alpha$-symmetric stable process into a map satisfying the Lagrange equation with respect to the fractional Laplacian. In this way, the fractional harmonic map discussed in \cite{Lio, Lio2} can be studied by using stochastic processes.

Next we will consider singularities of harmonic maps. Singularities of harmonic map heat flows with smooth initial data appearing at their explosion times were dealt with through stochastic processes in \cite{Thal1, Thal2}. In this paper, we will deal with singularities of harmonic maps in a weak sense including fractional harmonic maps based on \tref{harmonicmartingale}. In this situation, the continuity of martingales obtained from harmonic maps at time zero can be thought as the continuity of harmonic maps along the paths of Markov processes. This kind of continuity corresponds to the continuity in the fine topology in classical potential theory.
In \cite{Oka23}, it was shown that the equivalent condition of the fine continuity of quasi-harmonic maps for Markov processes with jumps can be described by quadratic variations of martingales on manifolds. By combining \tref{harmonicmartingale} above and Proposition 4.11 of \cite{Oka23}, we can obtain some equivalent conditions of the fine continuity of weakly harmonic maps with respect to general regular Dirichlet forms through martingales on manifolds. However, we note that the fine continuity is much weaker than the continuity in the topology of metric spaces in general. We will remark the equivalence does not hold for general weakly harmonic maps with respect to the Laplacian. One of the counterexamples is the everywhere-discontinuous weakly harmonic map constructed in \cite{Riv95}. However, if we consider a class of weakly harmonic maps with respect to the Laplacian and the fractional Laplacian which satisfy some assumptions regarding tangent maps, we can easily show that the equivalence holds. The precise assumptions and statement will be given in \asref{astangent} and \pref{minimizing}, respectively. Typical examples satisfying these assumptions are energy minimizing maps. 

We give an outline of the paper. In \sref{Markovprocesses}, we first recall some facts regarding symmetric Markov processes and Dirichlet forms. Mainly we refer to \cite{ChenFuku, FOT}. In addition, we recall the stochastic integral along continuous additive functionals of zero energy. In \sref{harmonicmap}, we characterize harmonic maps with respect to regular Dirichlet forms by martingales on manifolds. In \sref{pf2}, we describe the connection between sinularities of weakly harmonic maps and those of martingales.\\

Throughout this paper, given a topological space $M$, we denote by $C_0(M)$ the set of all continuous functions on $M$ with compact support. In the case that $M$ is a manifold, we denote by $C_0^{\infty}(M)$ the set of all $C^{\infty}$ functions with compact support. For $a, b \in \mathbb{R}$, we abbreviate $\max \{a,\, b\}$ and $\min \{a,\, b \}$ as $a\lor b$ and $a \land b$, respectively. For a stochastic process $H$ and a stopping time $\tau$, we write the stopped process as $H^{\tau}$ defined by
\[
H^{\tau}_t(\omega)=H_{t\land \tau (\omega)}(\omega).
\]
For two random variables $X$, $Y$, we write $X\sim Y$ if they have the same distribution. 
\section{Preliminaries on Markov processes}\label{Markovprocesses}
In this section, we recall the theory of Markov processes and Dirichlet forms in preparation for the proof of the main theorem. In particular, we focus on the stochastic integral along CAF's of zero energy introduced in \cite{Nakao} and further studied in \cite{CFK, Kuwae, Walsh}. As for the general theory of Dirichlet forms, see \cite{ChenFuku, FOT} for details.

Let $E$ be a locally compact and separable metric space. We add a point $\Delta$ to $E$ and define $E_{\Delta}:=E\cup \{\Delta \}$. Given $f\in \mathcal{B}(E)$, we extend the domain of $f$ to $E_{\Delta}$ by $f(\Delta)=0$. Let $m$ be a positive Radon measure with full support on $E$. Let $\mathcal{F}$ be a dense subspace of $L^2(E;m)$ and $\mathcal{E}\colon \mathcal{F} \times \mathcal{F} \to \mathbb{R}$ a non-negative definite symmetric quadratic form. We set
\[
\mathcal{E}_{\alpha}(u,v):=\mathcal{E}(u,v)+\alpha \langle u,v \rangle_{L^2}
\]
for $u,v\in \mathcal{F}$ and $\alpha >0$. The pair $(\mathcal{E}, \mathcal{F})$ is called a Dirichlet form if the space $(\mathcal{F}, \mathcal{E}_1)$ is a Hilbert space and it holds that $(u\lor 0)\land 1 \in \mathcal{F}$ and
\[
\mathcal{E}\left( (u\lor 0)\land 1, (u\lor 0)\land 1 \right)\leq \mathcal{E}(u,u)
\]
for any $u\in \mathcal{F}$. A Dirichlet form $(\mathcal{E},\mathcal{F})$ is said to be regular if $\mathcal{F}\cap C_0(E)$ is $\mathcal{E}_1$-dense in $\mathcal{F}$ and $\|\cdot \|_{\infty}$-dense in $C_0(E)$. 
For a regular Dirichlet form, the extended Dirichlet space $\mathcal{F}_e$ is defined as the family of equivalence classes of Borel functions $u\colon E\to \mathbb{R}$ with respect to the $m$-a.e. equality such that $|u|<\infty\ m$-a.e. and there exists an $\mathcal{E}$-Cauchy sequence $\{u_k\}_{k=1}^{\infty}$ in $\mathcal{F}$ such that $\displaystyle \lim_{k\to \infty}u_k=u,\ m$-a.e. For $u\in \mathcal{F}_e$, we can set
\[
\mathcal{E}(u,u):=\lim_{k\to \infty}\mathcal{E}(u_k,u_k),
\]
where $\{u_k\}_{k=1}^{\infty}$ is an $\mathcal{E}$-approximate sequence in $\mathcal{F}$ of $u$. For a regular Dirichlet form $(\mathcal{E},\mathcal{F})$, there exists an $m$-symmetric Hunt process $(\Omega, \{Z\}_{t\geq 0},\{ \theta_t\}_{t\geq 0}, \zeta, \{ \mathbb{P}_z\}_{z\in E_{\Delta}})$ on $E$ corresponding to $(\mathcal{E},\mathcal{F})$, where $\theta_t:\Omega \to \Omega$ is the shift operator satisfying for all $\omega \in \Omega$,
\begin{align*}
Z_s\circ \theta_t(\omega)&=Z_{s+t}(\omega),\ \theta_0\omega=\omega,\\
\theta_{\infty}\omega(t)&=\Delta,\ \text{for all}\ t\geq 0 ,
\end{align*}
$\zeta$ is a life time of $Z$ and $\mathbb{P}_z$ is the distribution of $\{Z_t\}$ stating at $z$. Set
\begin{align*}
\mathcal{F}^0_{\infty}&=\sigma(Z_s;\ s< \infty),\\
\mathcal{F}^0_t&=\sigma(Z_s;\ s\leq t).
\end{align*}
For a $\sigma$-finite measure $\mu$, we set
\[
\mathbb{P}_{\mu}(\Lambda)=\int_{E_{\Delta}}\mathbb{P}_z(\Lambda)\, \mu(dz),\ \Lambda \in \mathcal{F}^0_{\infty}.
\]
In particular, if $\mu$ is a probability measure on $E$, $\mathbb{P}_{\mu}$ is a probability measure on $\mathcal{F}^0_{\infty}$. Denote the $\mathbb{P}_{\mu}$-completion of $\mathcal{F}^0_{\infty}$ by $ \mathcal{F}^{\mu}_{\infty}$. Set
\[
\mathcal{F}^{\mu}_t=\sigma(\mathcal{F}^0_t,\mathcal{N}_{\mu}),
\]
where $\mathcal{N}_{\mu}$ is the family of all $\mathbb{P}_{\mu}$-null sets in $\mathcal{F}^{\mu}_{\infty}$. Denote the set of probability measures on $E_{\Delta}$ by $\mathcal{P}(E_{\Delta})$ and let
\[
\mathcal{F}^Z_t=\bigcap_{\mu \in \mathcal{P}(E_{\Delta})} \mathcal{F}^{\mu}_t,\ t\in [0,\infty].
\]
$\{ \mathcal{F}^Z_t \}_{t\geq 0}$ is called the minimum augmented admissible filtration of $Z$. A subset $A\subset E_{\Delta}$ is called a nearly Borel set if for every $\mu \in \mathcal{P}(E_{\Delta})$, there exist Borel sets $A_0$, $A_1$ in $E_{\Delta}$ such that
\[
A_0\subset A\subset A_1,\ \mathbb{P}_{\mu}(Z_t\in A_1\backslash A_0,\ \text{for some }t\geq 0)=0.
\]
For a subset $A\subset E_{\Delta}$, define the random time $\sigma_A$ by
\begin{align*}
\sigma_{A}(\omega)&:=\inf \{ t>0\mid Z_t(\omega)\in A\},\ \omega \in \Omega,\\
\tau_A(\omega)&:=\inf \{ t>0\mid Z_t(\omega )\notin A\},\ \omega \in \Omega.
\end{align*}
It is known that if $A$ is a nearly Borel set, $\sigma_A$ and $\tau_A$ are $\{ \mathcal{F}^Z_t \}$-stopping times.
A subset $A\subset E$ is said to be $m$-polar if there exists a nearly Borel set $A_2$ such that $A\subset A_2$ and
\begin{gather*}
\mathbb{P}_m(\sigma _{A_2}<\infty)=\int_E\mathbb{P}_z(\sigma_{A_2}<\infty)\, m(dz)=0.
\end{gather*}
The term q.e. stands for ``except for an $m$-polar set.'' A positive Radon measure $\mu$ on $E$ is said to be of finite energy integral if there exists $C>0$ such that
\[
\int_E|u(z)|\, \mu(dz)\leq C\sqrt{\mathcal{E}_1(u,u)},\ \text{for all}\ u\in \mathcal{F}\cap C_0(E).
\]
If $\mu$ is a positive Radon measure on $E$ of finite energy integral, then for each $\alpha >0$, there exists a unique function $U_{\alpha}\mu \in \mathcal{F}$ such that
\[
\mathcal{E}_{\alpha}(U_{\alpha}\mu,u)=\int_Eu(z)\, \mu(dz),\ \text{for all}\ u\in \mathcal{F}\cap C_0(E).
\]
$U_{\alpha}\mu$ is called an $\alpha$-potential. We denote by $S_0$ the family of all positive Radon measures of finite energy integrals. We set a subset $S_{00}$ of $S_0$ as
\begin{gather}\label{S00}
S_{00}:=\{ \mu \in S_0 \mid \mu(E)<\infty,\ \| U_1\mu \|_{\infty}<\infty \}.
\end{gather}
We use the abbreviation AF (resp. CAF) for additive functional (resp. continuous additive functional). Each PCAF $A$ determines a measure $\mu_A$ on $E$ called Revuz measure which is characterized by
\[ 
\lim_{t\to 0}\frac{1}{t}\mathbb{E}_m\left[\int_0^tf(Z_t)\, dA_t \right]=\int_Ef(z)\, d\mu_A(z),\ f\in \mathcal{B}_+(E),
 \]
where $\mathcal{B}_+(E)$ is the set of all non-negative Borel functions on $E$. We denote the energy of an AF $A$ by
\[
\mathbf{e}(A)=\lim_{t\searrow 0}\frac{1}{2t}\mathbb{E}_m\left[A_t^2 \right]
\]
and the mutual energy of AF's $A$, $B$ by
\[
\mathbf{e}(A,B)=\lim_{t\searrow 0}\frac{1}{2t}\mathbb{E}_m\left[A_tB_t \right] .
\]
We set
\begin{align*}
\mathcal{M}&:= \{ M \mid M\ \text{is an AF},\ \mathbb{E}_z\left[ M_t^2 \right] <\infty,\ \mathbb{E}_z\left[ M_t \right] =0,\ \text{for q.e.}\ z\in E,\ t\geq 0\},\\
\mathring{\mathcal{M}}&:= \{ M \in \mathcal{M} \mid \mu_{\langle M \rangle}(E) < \infty \},\\
\mathcal{N}_c&:= \{ N \mid N\ \text{is a CAF},\ \mathbf{e}(N)=0,\ \mathbb{E}_z\left[ |N_t| \right]<\infty,\ \text{q.e.}\ z  \}.
\end{align*}
Each additive functional in $\mathcal{M}$ is called a martingale additive functional (MAF). For details about MAF's, see \cite{FOT, Fuku}. For $u\in \mathcal{F}_e$, set $A^{[u]}_t:=\tilde u(X_t)-\tilde u(X_0)$, where $\tilde u$ is a quasi-continuous modification of $u$. Then $A^{[u]}$ is an additive functional of $Z$ with finite energy. Moreover, by \cite{FOT, Fuku}, there exists $M^{[u]}\in \mathring{\mathcal{M}}$ and $N^{[u]}\in \mathcal{N}_c$ such that
\[
A^{[u]}_t=M^{[u]}_t+N^{[u]}_t,\ \text{for all}\ t\geq 0,\ \mathbb{P}_z\text{-a.s. q.e.}\ z\in E 
\]
and such a decomposition is unique. For $u\in \mathcal{F}_e$, the MAF $M^{[u]}$ can be decomposed into continuous part, jump part, and killing part as follows. Since $M^{[u]}$ is a $P_z$-martingale for q.e. $z\in E$, the MAF $M^{[u]}$ can be written as the sum of its continuous martingale part $M^{[u],c}$ and purely discontinuous martingale part $M^{[u],d}$ and we can construct $M^{[u],c}$ and $M^{[u],d}$ as MAF's of $Z$. We further set
\begin{align*}
M^{[u],k}_t&:= -u(Z_{\zeta -})\mathbf{1}_{\{t\geq \zeta\}} + \int_0^t u(Z_s)\, N(Z_s, \Delta)\, dH_s,\\
M^{[u],j}_t&:= M^{[u],d}_t-M^{[u],k}_t,
\end{align*}
where $(N,H)$ is the L\'evy system of $Z$, which is defined by a pair of a kernel $N(z, dw)$ and a PCAF $H$ of $Z$ satisfying
\begin{align}\label{Levysystem}
\mathbb{E}_z \left[ \sum_{0<s\leq t}Y_sf(Z_{s-},Z_s) \right] =\mathbb{E}_z\left[ \int_0^{\infty}Y_s\left( \int_{E_{\Delta}}f(Z_s,w)\, N(Z_s,dw) \right) \, dH_s \right]
\end{align}
for any non-negative predictable process $\{Y_s\}$ and any $f\in \mathcal{B}_{+}(E_{\Delta} \times E_{\Delta})$ with $f(w,w)=0$ for all $w\in E$. Then $M^{[u],j}, M^{[u],k}\in \mathring{\mathcal{M}}$ and it holds that
\begin{align*}
&M^{[u]}=M^{[u],c}+M^{[u],j}+M^{[u],k},\\
&\langle M^{[u],c},M^{[u],j}\rangle = \langle M^{[u],j},M^{[u],k}\rangle = \langle M^{[u],k},M^{[u],c}\rangle = 0.
\end{align*}
We set
\[
J(dzdw):=N(z,dw)d\mu_H(dz),\ k(dz):=N(z,\Delta)d\mu_H(dz).
\]
Since it holds that 
\[
\mathcal{E}(u,v)=\mathbf{e}\left( A^{[u]},A^{[v]} \right)+\frac{1}{2}\int_Euv\, dk,\ u,v\in \mathcal{F}_e,
\]
the above decomposition of $M^{[u]}$ yields the following Beurling-Deny decomposition: For $u,v\in \mathcal{F}_e$,
\begin{align*}
\mathcal{E}(u,v)&=\frac{1}{2}\mu_{\langle M^{[u]}+M^{[u],k},M^{[v]}\rangle}(E)\\
&=\frac{1}{2}\mu^c_{\langle u,v \rangle} (E) + \frac{1}{2}\int_{E\times E\, \backslash \, \mathrm{diag}(E)} (u(z)-u(w))(v(z)-v(w))\, J(dzdw)\\
&\h + \int_E u(z)v(z) \, k(dz),
\end{align*}
where $\mu_{\langle u,v \rangle}^c$ is the Revuz measure of $\langle M^{[u],c},M^{[v],c}\rangle$ and $\mathrm{diag}(E)$ is a diagonal set of $E\times E$. We denote the family of purely discontinuous MAF's of finite energy by $\mathring{\mathcal{M}^d}$. In \cite{Kuwae}, it was shown that there is a one-to-one correspondence between $\mathring{\mathcal{M}^d}$ and the family of jump functions defined by
\begin{align*}
\mathring{\mathcal{J}}:=\{& \phi:E_{\Delta} \times E_{\Delta} \to [0,\infty) \mid \phi:\text{Borel measurable},\ \phi(z,z)=0\ \text{for all}\ z\in E_{\Delta},\\
&N(\phi^2)\mu_H\ \text{is a finite Radon measure which is null on each}\ m\text{-polar set} \}.
\end{align*}
For $\phi, \psi \in \mathring{\mathcal{J}}$, we denote $\phi \sim \psi$ if $\phi=\psi \ N(z,dw)\mu_H(dz)$-a.e. on $E\times E_{\Delta}$. Lemma 2.5 of \cite{Kuwae} guarantees that for all $\phi \in \mathring{\mathcal{J}}$, there exists $M\in \mathring{\mathcal{M}^d}$ such that
\[
M_t-M_{t-}=\phi (Z_{t-},Z_t),\ t\geq 0,\ P_z\text{-a.s. q.e.}\ z\in E
\]
and this is a one-to-one correspondence between $\mathring{\mathcal{J}}/\sim$ and $\mathring{\mathcal{M}^d}$. In particular, for $M\in \mathring{\mathcal{M}^d}$, there exists $K\in \mathring{\mathcal{M}^d}$ such that
\begin{gather}\label{jumpreverse}
K_t-K_{t-}=-\mathbf{1}_{E\times E}(\phi + \bar \phi) (Z_{t-},Z_t),
\end{gather}
where $\phi$ is the jump function corresponding to $M$ and $\bar \phi (z,w):=\phi (w,z)$ since $\phi + \bar \phi \in \mathring{\mathcal{J}}$.

Now we can define the stochastic integral along CAF's of zero energy. We set
\begin{align*}
\mathcal{N}_c^*:=& \left\{N + \int_0^{\cdot}g(Z_{s-})\, ds\mid N\in \mathcal{N}_c,\ g\in L^2(M;m) \right\},\\
\widetilde{\mathbb{A}}_c^+:=&\left\{A\mid \text{PCAF of}\ Z,\ \mu_A(E)<\infty \right\},\\
\widetilde{\mathcal{N}_c}:=& \left\{N + A - B\mid N \in \mathcal{N}_c^*,\ A,B\in \widetilde{\mathbb{A}}_c^+ \right\}.
\end{align*}
By Lemma 3.2 of \cite{Nakao}, we can define a linear operator $\gamma \colon \mathring{\mathcal{M}}\to \mathcal{F}$ and $\Gamma \colon \mathring{\mathcal{M}}\to \mathcal{N}_c^*$ as follows: For each $M\in \mathring{\mathcal{M}}$, there exists a unique $w\in \mathcal{F}$ such that
\[
\mathcal{E}_1(w,u)=\frac{1}{2}\mu_{\langle M^{[u]}+M^{[u],k},M\rangle}(E)\ \text{for all}\ u\in \mathcal{F}.
\]
This $w$ is denoted by $\gamma (M)$ and $\Gamma$ is defined by
\[
\Gamma (M)_t:= N^{[\gamma (M)]}_t-\int_0^t \gamma (M)(Z_s)\ ds,\ \text{for each}\ M\in \mathring{\mathcal{M}}.
\]
Let $M\in \mathring{\mathcal{M}}$ and $\phi \in \mathring{\mathcal{J}}$ its jump function. For $g\in \mathcal{F}_e\cap L^2(E;\mu_{\langle M \rangle})$, the stochastic integral $\int g(Z)\, d\Gamma (M) \in \widetilde{\mathcal{N}_c}$ is defined by
\[
\int_0^tg(Z_{s-})\, d\Gamma (M)_s:=\Gamma (g*M)-\frac{1}{2}\langle M^{[g]}+M^{[g],k},M^c+M^j+K \rangle,\ t\in [0,\infty ),
\]
where $(g* M)_t=\int_0^t g(Z_{s-})\, dM_s$ and $K\in \mathring{\mathcal{M}}^d$ constructed by \eqref{jumpreverse} (cf. (3.4) of \cite{Kuwae}). By \cite{Nakao}, the CAF $\Gamma (M)$ can be characterized by
\begin{gather}
\lim_{t\to 0}\frac{1}{t}\mathbb{E}_{lm}\left[ \Gamma (M)_t\right]=-\frac{1}{2}\mu_{\langle M^{[l]}+M^{[l],k}, M\rangle}(E)\ \text{for all}\ l\in \mathcal{F}_b,\label{Nakaooperator}
\end{gather}
where $\mathcal{F}_b=\mathcal{F}\cap L^{\infty}(E;m)$.
\begin{lem}
For $M\in \mathring{\mathcal{M}}$ and $g, l\in \mathcal{F}_b$, it holds that
\begin{gather}
\lim_{t\to 0}\frac{1}{t}\mathbb{E}_{lm}\left[ \int_0^tg(Z_{s-})\, d\Gamma (M)_s \right]=-\frac{1}{2}\mu_{\langle M^{[lg]}+M^{[lg],k}, M\rangle}(E).\label{rightdel}
\end{gather}
\end{lem}
\begin{proof}
By derivation property, it holds that
\[
d\mu_{\langle M^{[lg],c},M^c\rangle}(z)=\tilde l (z)\, d\mu_{\langle M^{[g],c},M^c\rangle }(z)+\tilde g (z) \, d\mu_{\langle M^{[l],c},M^c\rangle}(z).
\]
Moreover, in view of Lemma 3.1 of \cite{Nakao}, we have
\begin{align*}
\mu_{\langle M^{[lg],j},M^j\rangle}(E)&=\int_E \tilde l\, d \mu_{\langle M^{[g],j},M^j+K\rangle}+\int_E \tilde g\, d\mu_{\langle M^{[l],j},M^j \rangle},\\
d\mu_{\langle M^{[lg],k},M^k\rangle}(z)&=\tilde l (z)\, d\mu_{\langle M^{[g],k},M^k\rangle}(z). 
\end{align*}
Therefore
\begin{align*}
\lim_{t\to 0}\frac{1}{t}\mathbb{E}_{lm}\left[ \Gamma (g*M)_t\right]&=-\frac{1}{2}\mu_{\langle M^{[l]}+M^{[l],k}, g*M \rangle}(E)\\
&= -\frac{1}{2}\int_Eg(z)\, \mu_{\langle M^{[l]}+M^{[l],k}, M \rangle}(dz)\\
&= -\frac{1}{2}\left( \int_Eg(z)\, \mu_{\langle M^{[l],c}, M^c \rangle}(dz)\right. + \int_Eg(z)\, \mu_{\langle M^{[l],j}, M^j \rangle}(dz)\\
&\h +\left. 2\int_Eg(z)\, \mu_{\langle M^{[l],k}, M^k \rangle}(dz)\right) \\
&= -\frac{1}{2}\left( \mu_{\langle M^{[lg],c},M^c\rangle}(E)-\int_E l\, d\mu_{\langle M^{[g],c},M^c \rangle}+\mu_{\langle M^{[lg],j},M^j\rangle}(E) \right. \\
&\h \left. -\int_E ld\mu_{\langle M^{[g],j},M^j+K\rangle} +2\mu_{\langle M^{[lg],k},M^k \rangle}(E) \right)\\
&= -\frac{1}{2}\mu_{\langle M^{[lg]}+M^{[lg],k},M\rangle}(E)+\frac{1}{2}\int_El\, d\mu_{\langle M^{[g],c}+M^{[g],j},M^c+M^j+K\rangle}.
\end{align*}
Thus we obtain \eqref{rightdel}.
\end{proof}
For an open set $D\subset E$, we set
\begin{align*}
\mathcal{F}^D&:=\{ u\in \mathcal{F} \mid \tilde u=0,\ \text{q.e. on}\ E\, \backslash \, D\},\\
\mathcal{F}^D_{loc}&:=\{u:E\to \mathbb{R}\mid \text{for every relatively compact open subset}\ D_1\subset D,\\
&\hspace{2.7cm} \text{there exists}\ g\in \mathcal{F}^D\ \text{such that}\ u=g,\ m\text{-a.e. on}\ D_1\}.
\end{align*}
We will define the stochastic integral $\int_0^{t\land \tau_{D_1}} g(Z_-)\, d\Gamma (M)$, where $D_1\subset D$ is a relatively compact open set such that $\overline{D_1}\subset D$, $g$ is a locally bounded function in $\mathcal{F}^D_{loc}$ and $M\in \mathring{\mathcal{M}}$.
\begin{lem}\label{integralPCAF}
Let $M\in \mathring{\mathcal{M}}$ and $D$ a relatively compact open set of $E$. Suppose that there exists a  CAF $A$ of bounded variation of $Z$ such that $|\mu_A|(\overline{D})<\infty$ and
\[
\Gamma (M) = A,\ t\leq \tau_{D}.
\]
Then for $g\in \mathcal{F}_b$, it holds that
\[
\int _0^{t\wedge \tau_D}g(Z_{s-})\, d\Gamma (M)_s=\int_0^{t\wedge \tau_D}g(Z_{s-})\, dA_s,\ \mathbb{P}_z\text{-a.s. q.e.}\ z\in D,
\]
where the right hand side is the Lebesgue-Stieltjes integral along $A$.
\end{lem}
\begin{proof}
Take any $l\in \mathcal{F}^D_b$. Then by \eqref{Nakaooperator}, it holds that
\begin{align*}
\lim_{t\to 0}\frac{1}{t}\mathbb{E}_{lgm}\left[ \Gamma (M)_t\right]=-\frac{1}{2}\mu_{\langle M^{[lg]}+M^{[lg],k},M \rangle}(E).
\end{align*}
On the other hand, it holds that
\[
\lim_{t\to 0}\frac{1}{t}\mathbb{E}_{l gm}\left[ \Gamma (M)_t\right]=\int_E (l g) (z) \mu_A(dz)
\]
by Lemma 5.4.4 of \cite{FOT}, Theorem 2.2 of \cite{Nakao}, and $|\mu_A|(\overline{D})<\infty$. Thus we obtain
\begin{align*}
\lim _{t\to 0}\frac{1}{t}\mathbb{E}_{lm}\left[ \int _0^t g(Z_{s-})\, d\Gamma (M)_s \right]= \int_E (lg) (z) \mu_A(dz).
\end{align*}
Therefore by Lemma 5.4.4 of \cite{FOT} and Theorem 2.2 of \cite{Nakao}, it holds that
\begin{gather*}
\int _0^{t\wedge \tau_{D}}g(Z_{s-})\, d\Gamma (M)_s=\int_0^{t\wedge \tau_{D}}g(Z_{s-})\, dA_s.
\end{gather*}
Thus we obtain the desired assertion.
\end{proof}
For an open set $D\subset E$, we set
\begin{gather*}
\mathcal{E}^D(u,v):=\mathcal{E}(u,v),\ \text{for}\ u,v\in \mathcal{F}^D.
\end{gather*}
Then $(\mathcal{E}^D,\mathcal{F}^D)$ is a regular Dirichlet form on $L^2(D; \mathbf{1}_D m)$. The symmetric Hunt process $Z^D$ corresponding to $(\mathcal{E}^D,\mathcal{F}^D)$ can be obtained as follows: We let for $\omega \in \Omega$,
\begin{gather*}
Z_t^D:=
\begin{cases}
Z_t(\omega),\ &0\leq t< \tau_D(\omega),\\
\Delta,\ &t\geq \tau_D(\omega),
\end{cases}
\end{gather*}
and $\zeta^D(\omega):=\tau_D(\omega)$. Define the shift operator $\theta^D_t$ on $\Omega$ by
\begin{gather*}
\theta^D_t\omega:=
\begin{cases}
\theta_t \omega,\ &t<\tau_D(\omega),\\
\omega_{\Delta},\ &t\geq \tau_D(\omega),
\end{cases}
\end{gather*}
where $\omega_{\Delta} \in \Omega$ such that $Z_t(\omega_{\Delta})=\Delta$ for all $t\geq 0$. Then $(\Omega, \{Z_t^D\}_{t\geq 0}, \{\theta^D_t\}_{t\geq 0}, \{\mathbb{P}_z\}_{z\in D_{\Delta}})$ is a symmetric Hunt process on $(D,\mathcal{B}(D))$ corresponding to $(\mathcal{E}^D,\mathcal{F}^D)$. We set
\begin{align*}
\mathcal{G}_t^0&:=\sigma (Z^D_s;s\leq t),\\
\mathcal{G}_t^{\mu}&:=\sigma (\mathcal{G}_t^0,\mathcal{N}_{\mu}),\ \text{for}\ \mu \in \mathcal{P}(D_{\Delta}),\\
\mathcal{G}_t&:=\bigcap_{\mu \in \mathcal{P}(D_{\Delta})}\mathcal{G}_t^{\mu}.
\end{align*}
Then the inclusion $\mathcal{F}^Z_{t\land \tau_D-}\subset \mathcal{G}_t\subset \mathcal{F}^Z_{t\land \tau_D}$ holds. If $A$ is an AF of $Z$, $B_t:=A_{t\land \tau_D}$ is $\{\mathcal{F}_{t\land \tau_{D}}\}$-adapted and satisfies the additivity
\[
B_{s+t}=B_s+B_t\circ \theta^D_s.
\]
Moreover, if $A$ is a CAF of $Z$, then $B_t$ is $\{\mathcal{G}_t\}$-adapted and a CAF of $Z^D$.
\begin{dfn}
Let $M\in \mathring{\mathcal{M}}$, $D$ an open set of $E$ and $D_1$ a relatively compact open set of $D$ such that $\overline{D_1}\subset D$. The integral of locally bounded function $g$ in $\mathcal{F}_{loc}^D$ along $\Gamma (M)^{\tau_{D_1}}$ is defined by
\begin{gather}
\int_0^t g(Z_{s-})\, d\Gamma (M)^{\tau_{D_1}}:= \int_0^{t\land \tau_{D_1}} (\phi_{D_2} g)(Z_{s-})\, d\Gamma (M),\ t\geq 0,\label{integralexitdef}
\end{gather}
where $D_2$ is a relatively compact open neighborhood of $D$ such that $\overline{D_1}\subset D_2 \subset \overline{D_2} \subset D$ and $\phi_{D_2}$ is a function satisfying $\phi_{D_2}\in C_0(E)\cap \mathcal{F}$ with $0\leq \phi_{D_2} \leq 1$, $\phi_{D_2}=1$ on $D_2$.
\end{dfn}
\begin{rem}
This integral is independent of the choice of $D_2$ and $\phi_{D_2}$ by Lemma 3.4 of \cite{Kuwae}. Moreover, it is also independent of the choice of $L\in \mathring{\mathcal{M}}$ satisfying
\[
\Gamma (M)^{\tau_{D_1}}=\Gamma (L)^{\tau_{D_1}}.
\]
In fact, we have $\Gamma (M-L)^{\tau_{D_1}}=0$ and consequently it holds that
\[
\int g(Z_{s-})\, d\Gamma (M)^{\tau_{D_1}}=\int g(Z_{s-})\, d\Gamma (L)^{\tau_{D_1}}
\]
by \lref{integralPCAF}. Thus we can deduce that the integral \eqref{integralexitdef} is well-defined.
\end{rem}
\begin{dfn}\label{integralexittime}
Let $H$ be a function on $[0,\infty)\times \Omega$ and $M\in \mathring{\mathcal{M}}$. Let $J=H+\Gamma (M)$. Suppose that the process $t\to H_{t\land \tau_{D_1}}$ is a $\mathbb{P}_z$-local mimartingale for q.e. $z\in E$. Then for $g\in \mathcal{F}^D_{loc}$ the stochastic integral $\int g(Z_{s-})\, dJ^{\tau_{D_1}}$ is defined by
\[
\int_0^t g(Z_{s-})\, dJ^{\tau_{D_1}}_s:= \int_0^t g(Z_{s-})\, dH^{\tau_{D_1}}_s + \int_0^t g(Z_{s-})\, d\Gamma (M)^{\tau_{D_1}}_s,\ t\leq \tau_{D_1}.
\]
\end{dfn}
\begin{rem}\label{integralexittimerem}
The integral is well-defined, i.e. the right-hand side is independent of the choice of a function $H'\colon [0,\infty)\times \Omega \to \mathbb{R}$ and $M'\in \mathring{\mathcal{M}}$ such that $H'^{\tau_{D_1}}$ is a $\mathbb{P}_z$-local martingale for q.e. $z\in E$ and
\[
H^{\tau_{D_1}}+\Gamma (M)^{\tau_{D_1}} = (H')^{\tau_{D_1}}+\Gamma (M')^{\tau_{D_1}}.
\]
In fact, we have $H^{\tau_{D_1}}-(H')^{\tau_{D_1}} = \Gamma (L)^{\tau_{D_1}}$, where $L=M'-M$, and consequently,   $H^{\tau_{D_1}}-(H')^{\tau_{D_1}}$ is a CAF of $Z^{D_1}$ which is a $\mathbb{P}_z$-local martingale for q.e. $z \in E$. For a stochastic process $Q_t$, define
\[
V_Q^n(t):=\sum_{k=1}^n\left( Q_{\left( \frac{k}{n}t\right) \wedge \tau_D}-Q_{\left( \frac{k-1}{n}t\right) \wedge \tau_D}\right) ^2.
\]
Then for all $N \in \widetilde{\mathcal{N}_c}$,
\begin{align*}
\mathbb{E}_m\left[ V^n_N(t)\right] & \leq \mathbb{E}_m\left[ \sum_{k=1}^n\left( N_{\frac{kt}{n}}-N_{\frac{(k-1)t}{n}} \right)^2 \right]\\
&= \sum_{k=1}^n\mathbb{E}_m\left[ \left( N_{\frac{t}{n}}\circ \theta_{\frac{k-1}{n}t}\right)^2 \right]\\
&\leq n\mathbb{E}_m\left[ \left(N_{\frac{t}{n}} \right)^2\right] \to 0\ (\text{as}\ n\to \infty).
\end{align*}
On the other hand, we can define the quadratic variation of $\Gamma(L)^{\tau_{D_1}}$ as a $\mathbb{P}_z$-continuous local martingale for q.e. $z\in E$ and it holds that $\langle \Gamma(L)^{\tau_{D_1}} \rangle=0$ $\mathbb{P}_z$-a.s. for $m$-a.e. $z \in D_1$ by Fatou's lemma. Thus $\Gamma (L)^{\tau_{D_1}}$ is $m$-equivalent to zero as a CAF of $Z^{D_1}$. Therefore we can deduce that
\[
\Gamma(M)_{t\land \tau_{D_1}}=\Gamma(M')_{t\land \tau_{D_1}},\ t\geq 0,\ \mathbb{P}_z\text{-a.s. q.e.}\ z\in D_1.
\]
For $z\in E\, \backslash \, D_1$, it is obvious that $\Gamma (M)^{\tau_{D_1}} = \Gamma(M')^{\tau_{D_1}}=0$, $\mathbb{P}_z$-a.s. Thus the decomposition of $J^{\tau_{D_1}}$ is unique and consequently the integral is well-defined.
\end{rem}
\section{Proof of \tref{harmonicmartingale}}\label{harmonicmap}
First we define martingales and harmonic maps with respect to regular Dirichlet forms. The theory of discontinuous martingales on manifolds was developed in \cite{Pic1991}. In \cite{Oka23}, the author focused on discontinuous martingales on Riemannian submanifolds of higher dimensional Euclidean spaces as a special case considered in \cite{Pic1991}, and then extended them so that we allowed the killing of martingales. We begin with recalling semimartingales and martingales defined in \cite{Oka23}. Let $M$ be a Riemannian submanifold of $\mathbb{R}^d$.
\begin{dfn}\label{semimartingale}
Let $X$ be an $\mathbb{R}^d$-valued semimartingale, $\zeta$ an $\{\mathcal{F}_t\}_{t\geq 0}$-stopping time and $p$ a point in $\mathbb{R}^d$. We call $(X,\zeta, p)$ an $M$-valued semimartingale with the the end point $p$ if it satisfies $X_t\in M$ for $t\in [0,\zeta)$ and $X_t=p$ for $t \geq \zeta$ a.s.
\end{dfn}
If $p \in M$ or $\zeta=\infty$ almost surely, then an $M$-valued semimartingale $X$ with the end point $p$ is just an $M$-valued semimartingale in a usual sense.
\begin{dfn}
Let $(X,\zeta, p)$ be an $M$-valued $\{\mathcal{F}_t\}_{t\geq 0}$-semimartingale with the end point $p$. The triple $(X,\zeta, p)$ is called an $M$-valued martingale with the end point $p$ if for any smooth vector field $V$ on $M$, the stochastic integral $\displaystyle \int \langle V(X_-), dX \rangle$ is a local martingale.
\end{dfn}
\begin{dfn}\label{defquasi-harmonic}
A Borel measurable map $u:E\to M$ is said to be quasi-harmonic on $D$ if for each relatively compact open set $D_1$ such that $\overline{D_1}\subset D$, $(u(Z)^{\tau_{D_1}},\zeta, 0)$ is an $M$-valued $(\mathbb{P}_z, \{ \mathcal{F}^Z_t \})$-martingale with an end point for q.e. $z\in E$.
\end{dfn}
In \dref{defquasi-harmonic}, we set $0\in \mathbb{R}^d$ as the end point since we set the value of a function on $E_{\Delta}$ at $\Delta$ as $0$. If $Z$ has no inside killing, $u$ is quasi-harmonic on $D$ if and only if for each relatively compact open set $D_1$ such that $\overline{D_1}\subset D$, $\{u(Z_{t \land \tau_{D_1}})\}_{t\geq 0}$ is an $M$-valued $\Prob_z$-martingale for q.e. $z\in E$.\\

Next we will recall the following conditions for a function $u\in \mathcal{F}^D_{loc}$ which were considered in \cite{Chen} and Chapter 6 of \cite{ChenFuku}: For any relatively compact open set $D_1$, $D_2$ with
\begin{align}\label{opensets}
\overline{D_1}\subset D_2 \subset \overline{D_2} \subset D,
\end{align}
\begin{description}
\item[\thetag{A}]$\int_{D_1\times (E \backslash \, D_2)}|u(w)|J(dzdw) <\infty$
\item[\thetag{B}]and if we define a function $f_u$ by
\begin{align}
f_u(z):= \mathbf{1}_{D_1}(z)\mathbb{E}_z\left[ ((1-\phi_{D_2})|u|)(Z_{\tau_{D_1}})\right]\ \text{for}\ z\in E,\label{fu}
\end{align}
then $f_u \in \mathcal{F}_e^{D_1}$, where $\phi_{D_2}$ is a function satisfying
\[
\phi_{D_2}\in \mathcal{F}\cap C_0(D),\ 0\leq \phi_{D_2} \leq 1,\ \phi_{D_2}=1\ \text{on}\ D_2.
\]
\end{description}
For a Borel measurable locally bounded function $u\in \mathcal{F}^D_{loc}$ satisfying \thetag{A}, we can define $\mathcal{E}(u,\psi)$ for all $\psi \in C_0(D)\cap \mathcal{F}$ by
\begin{align}
\mathcal{E}(u,\psi)&=\frac{1}{2}\mu_{\langle u,\psi \rangle}^c(D) + \frac{1}{2}\int_{E\times E \, \backslash \, \mathrm{diag}(E)}(u(z)-u(w))(\psi (z) - \psi (w) )\, J(dzdw) \nonumber \\
&\h + \int _D u(z) \psi (z) \, k(dz).\label{Eupsi}
\end{align}
The right-hand side of \eqref{Eupsi} is finite by Lemma 6.7.8 of \cite{ChenFuku}. For a Borel measurable function $u:E\to \mathbb{R}$, we set
\begin{align*}
C_t&:=\int_0^{t\land \tau_{D_1}}\int_{E\, \backslash \, D_2}(1-\phi_{D_2})|u|(w)\, N(Z_s,dw)dH_s,\\
C_t^+&:=\int_0^{t\land \tau_{D_1}}\int_{E\, \backslash \, D_2}(1-\phi_{D_2})u^+(w)\, N(Z_s,dw)dH_s,\\
C_t^-&:=\int_0^{t\land \tau_{D_1}}\int_{E\, \backslash \, D_2}(1-\phi_{D_2})u^-(w)\, N(Z_s,dw)dH_s
\end{align*}
with $u^+=u\lor 0,\ u^-=(-u)\lor 0$. Then $C_t,\ C_t^+,\ C_t^-$ are PCAF's of $Z^{D_1}$ and their Revuz measures are
\begin{align}
\mu_u(dz)&:=\mu_{C}(dz)=\mathbf{1}_{D_1}(z)\int_{E\, \backslash \, D_2}(1-\phi)|u|(w)\, N(z,dw)d\mu_H(dz),\\
\mu^+_u(dz)&:=\mu_{C^+}(dz)=\mathbf{1}_{D_1}(z)\int_{E\, \backslash \, D_2}(1-\phi)u^+(w)\, N(z,dw)d\mu_H(dz),\\
\mu^-_u(dz)&:=\mu_{C^-}(dz)=\mathbf{1}_{D_1}(z)\int_{E\, \backslash \, D_2}(1-\phi)u^-(w)\, N(z,dw)d\mu_H(dz).
\end{align}
\begin{lem}\label{fu+fu-}
Let $u\in \mathcal{F}^D_{loc}$ be a locally bounded Borel measurable function satisfying \thetag{A} and \thetag{B}. Then it holds that $\mu_u(D_1)<\infty$ and $f_{u^+}, f_{u^-}\in \mathcal{F}_e^{D_1}$, where
\begin{align*}
f_{u^+}(z)&=\mathbf{1}_{D_1}(z)\mathbb{E}_z\left[ ((1-\phi_{D_2})u^+)(Z_{\tau_{D_1}})\right],\\
f_{u^-}(z)&=\mathbf{1}_{D_1}(z)\mathbb{E}_z\left[((1-\phi_{D_2})u^-)(Z_{\tau_{D_1}})\right].
\end{align*}
\end{lem}
\lref{fu+fu-} follows by applying Lemma 6.7.6 of \cite{ChenFuku} to $u^+$ and $u^-$. 
\begin{lem}\label{hAB}
Let $D,D_1,D_2$ be open sets in $E$ satisfying \eqref{opensets} and $u\in \mathcal{F}^D_{loc}$ a locally bounded Borel measurable function satisfying \thetag{A} and \thetag{B}. Then the process
\[
A_t:= \mathbf{1}_{\{t\geq \tau_{D_1}\}}\{(1-\phi (Z_{\tau_{D_1}}))u(Z_{\tau_{D_1}})-(1-\phi (Z_{\tau_{D_1}-}))u(Z_{\tau_{D_1}-})\}
\]
is of $\mathbb{P}_z$-integrable variation for q.e. $z\in E$ and the dual predictable projection of $A_t$ is
\[
B_t=\int_0^{t\land \tau_{D_1}}\int_{E\, \backslash \, D_2}(1-\phi)u(z)\, N(Z_s,dz)dH_s=C^+_t-C^-_t.
\]
We further set
\begin{align*}
h_1(z)&:=\mathbb{E}_z\left[ (\phi u)(Z_{\tau_{D_1}}) \right],\\
h_2(z)&:=\mathbb{E}_z \left[((1-\phi )u)(Z_{\tau_{D_1}}) \right].
\end{align*}
Then it holds that $\phi u \in \mathcal{F}^D,\ h_1 \in \mathcal{F}_e,\ \phi u-h_1\in \mathcal{F}_e^{D_1}$, and $\mathbf{1}_{D_1} h_2\in \mathcal{F}_e^{D_1}$. Moreover $h_2=\mathbf{1}_{D_1}h_2+(1-\phi)u$ satisfies \thetag{A} and it holds that
\[
\mathcal{E}(h_2,\psi)=0,\ \text{for all}\ \psi \in C_0(D_1)\cap \mathcal{F}.
\]
\end{lem}
\begin{proof}
First we set
\begin{gather*}
A^+_t:=\mathbf{1}_{\{t\geq \tau_{D_1}\}} \{(1-\phi (Z_{\tau_{D_1}}))u^+(Z_{\tau_{D_1}})-(1-\phi (Z_{\tau_{D_1}-}))u^+(Z_{\tau_{D_1}-})\}\\
A^-_t:=\mathbf{1}_{\{t\geq \tau_{D_1}\}} \{(1-\phi (Z_{\tau_{D_1}}))u^-(Z_{\tau_{D_1}})-(1-\phi (Z_{\tau_{D_1}-}))u^-(Z_{\tau_{D_1}-})\}.
\end{gather*}
Then by L\'evy system formula \eqref{Levysystem}, it holds that
\begin{gather*}
\mathbb{E}_z\left[ A^+_{\tau_{D_1}} \right]=\mathbb{E}_z\left[C^+_{\tau_{D_1}}\right]=f_{u^+}(z)<\infty,\\
\mathbb{E}_z\left[ A^-_{\tau_{D_1}}\right]=\mathbb{E}_z\left[C^-_{\tau_{D_1}}\right]=f_{u^-}(z)<\infty
\end{gather*}
for q.e. $z\in E$. Thus we obtain
\[
\mathbb{E}_z\left[|A_{\tau_{D_1}}|\right]=f_{u^+}(z)+f_{u^-}(z)<\infty,\ \text{for q.e.}\ z\in E.
\]
Thus $A_t$ is a process of integrable variation and its dual predictable projection is $B_t$ under $\mathbb{P}_z$ for q.e. $z\in E$. In the second claim of \lref{hAB}, $\phi u\in \mathcal{F}^D$, $h_1\in \mathcal{F}_e$, and $\phi u-h_1\in \mathcal{F}_e^{D_1}$ are obvious from the decomposition of $\mathcal{F}_e^D$. Moreover, we have $\mathbf{1}_{D_1}h_2=f_{u^+}-f_{u^-}\in \mathcal{F}_e^{D_1}$ by \lref{fu+fu-}. The last claim also can be shown in the same way as the proof of Theorem 6.7.9 of \cite{ChenFuku} without the harmonicity of $u$.
\end{proof}
Let $\mathcal{C}$ be a special standard core of $(\mathcal{E}, \mathcal{F})$ and $D$ an open set of $E$. We set
\begin{align*}
\mathcal{C}_D&:=\{ \psi \in \mathcal{C}\mid \mathrm{supp}[\psi]\subset D \},\\
\mathcal{C}_D(\mathbb{R}^d)&:=\{ \psi:E\to \mathbb{R}^d \mid \psi=(\psi^1,\dots,\psi^d),\ \psi^i\in \mathcal{C}_D\ \text{for each }i\},\\
\mathcal{F}^D_{loc}(\mathbb{R}^d)&:=\{ u=(u^1,\dots,u^d):E\to \mathbb{R}^d\mid u^i \in \mathcal{F}^D_{loc}\ \text{for each}\ i\},\\
\mathcal{F}^D_{loc}(M)&:=\{ u\in \mathcal{F}_{loc}^D(\mathbb{R}^d) \mid u(z)\in M,\ m\text{-a.e. }z\}.
\end{align*}
$u=(u^1,\dots,u^d)\in \mathcal{F}^D_{loc}(M)$ is said to satisfy \thetag{A} (resp. \thetag{B}) if $u^i$ satisfies \thetag{A} (resp. \thetag{B}) for each $i$. For $u\in \mathcal{F}^D_{loc}(M)$, we further set
\[
\mathcal{C}_D(u^*TM):=\{ \psi \in \mathcal{C}_D(\mathbb{R}^d)\mid \psi(z)\in T_{u(z)}M,\ \text{a.e.}\ z\in D \}.
\]
We define $\mathcal{F}^D(\mathbb{R}^d)$, $\mathcal{F}^D(u^*TM)$ and $\mathcal{F}^D_e(u^*TM)$ in the same way.
\begin{dfn}\label{weaklyharmonic}
Let $u\in \mathcal{F}^D_{loc}(M)$ be a locally bounded Borel measurable map satisfying \thetag{A} and \thetag{B}. $u$ is called a weakly harmonic map on $D$ if
\[
\sum_{j=1}^d\mathcal{E}(u^j,\psi^j)=0,\ \text{for all}\ \psi \in \mathcal{C}_D(u^*TM).
\]
\end{dfn}
We divide \tref{harmonicmartingale} into the following Theorems \ref{harmonicmartingale1} and \ref{harmonicmartingale2}.
\begin{thm}\label{harmonicmartingale1}
Let $u\colon E \to M$ be a Borel measurable map in $\mathcal{F}^D_{loc}(M)$ which is locally bounded on $D$ and satisfies \thetag{A} and \thetag{B}. Suppose that $u$ is weakly harmonic on $D$. Then $u$ is quasi-harmonic on $D$.
\end{thm}
\begin{thm}\label{harmonicmartingale2}
Let $u \in \mathcal{F}^D_{loc}(M)$ be quasi-harmonic on $D$, locally bounded on $D$ and satisfy \thetag{A} and \thetag{B}. Then $u$ is weakly harmonic on $D$.
\end{thm}
\begin{proof}[Proof of \tref{harmonicmartingale1}]
Let $\phi=\phi_{D_2}$, $h_1(z):=\mathbb{E}_z\left[ (\phi u)(Z_{\tau_{D_1}})\right]$, $h_2(z):=\mathbb{E}_z\left[ ((1-\phi)u)(Z_{\tau_{D_1}})\right]$, $h=h_1+h_2$ and $v:=u-h$. Then for each $j=1,\dots,d$, $v^j=(\phi u^j-h^j_1) - \mathbf{1}_{D_1}h^j_2\in \mathcal{F}_e^{D_1}$. Thus we can decompose $u^j(Z_t)$ as
\begin{gather}
u^j(Z_t)=u^j(Z_0)+h^j(Z_t)-h^j(Z_0)+M^{[v^j]}_t+N_t^{[v^j]}\label{uhv}
\end{gather}
by Fukushima decomposition. Furthermore, since $\mathcal{E}(h^j,\psi^j)=0$ for all $\psi \in \mathcal{C}_{D_1}(\mathbb{R}^d)$ and $j=1,\dots,d$ by \lref{hAB}, it holds that
\begin{align}
\sum_{j=1}^d\mathcal{E}(v^j,\psi^j)=\sum_{j=1}^d\mathcal{E}(u^j,\psi^j)-\sum_{j=1}^d\mathcal{E}(h^j,\psi^j)=0.\label{Evpsi}
\end{align}
Note that $\mathcal{C}_D(u^*TM)$ is $\mathcal{E}$-dense in $\mathcal{F}_e^D(u^*TM)$ because $\mathcal{F}^D(u^*TM)$ is a closed subset in $(\mathcal{F}^D(\mathbb{R}^d),\mathcal{E}_1^D)$. Then \eqref{Evpsi} holds for all $\psi \in \mathcal{F}^{D_1}_e(u^*TM)$. We take $f\in C^{\infty}(M)$ and extend $f$ to $\bar{f} \in C^{\infty}(\mathbb{R}^d)$ in such a way that for $x\in M$,
\begin{gather}\label{nabla}
\bar{f}(x)=f(x),\ D\bar{f}(x)=\nabla f(x),
\end{gather}
where $D\bar{f}$ is the gradient of $\bar{f}$ as a function on $\mathbb{R}^d$. Then $\phi D\bar{f}\circ u \in \mathcal{F}^{D}(u^*TM)$. It suffices to show that $u(Z)^{\tau_{D_1}}$ is a $\mathbb{P}_z$-semimartingale and $\int_0^t D_j\bar{f}\circ u (Z_{s-})\, du(Z)^{\tau_{D_1}}_s$ is a $\mathbb{P}_z$-local martingale for q.e. $z\in E$. Since $H_t:=h(Z_{t\land \tau_{D_1}})-h(Z_0)$ is an $\mathbb{R}^d$-valued $\mathbb{P}_z$-uniformly integrable martingale for q.e. $z\in E$, we can define the stochastic integral $\int_0^t D_j\bar{f}\circ u (Z_{s-})\, du(Z)^{\tau_{D_1}}_s$ in the sense of \dref{integralexittime}. Then it holds that
\begin{align*}
\lim_{t\to 0}\frac{1}{t}\mathbb{E}_{lm}\left[\sum_{j=1}^d \int_0^t \phi D_j\bar{f}\circ u(Z_{s-})dN^{[v^j]}_s\right] = -\sum_{j=1}^d \mathcal{E}(l \phi D_j\bar{f}\circ u,v^j) = 0
\end{align*}
for all $l\in \mathcal{F}^{D_1}_b$ since $l \phi D\bar{f}\circ u \in \mathcal{F}^{D_1}(u^*TM)$. Therefore
\begin{align}
\sum_{j=1}^d\int_0^t D_j\bar{f}\circ u(Z_{s-})dN_s^{[v^j],\, \tau_{D_1}}=0,\ t\geq 0\label{intN}
\end{align}
holds and $\dis \sum_{j=1}^d \int_0^t D_j\bar{f}\circ u(Z_{s-})\, du^j(Z)^{\tau_{D_1}}_s$ is a $\mathbb{P}_z$-local martingale for q.e. $z\in E$. Next we will show that $u(Z_{t\wedge \tau_{D1}})$ is a semimartingale. Since it holds that
\[
f(u(Z_{t\wedge \tau_{D_1}})) = f((\phi u)(Z_{t\land \tau_{D_1}}))+\left\{ f(u (Z_{\tau_{D_1}})) - f((\phi u) (Z_{\tau_{D_1}})) \right\} \mathbf{1}_{\{ t\geq \tau_{D_1}\}}
\]
and the second term of the right-hand side is a process of bounded variation, it suffices to show that $f(\phi u(Z))^{\tau_{D_1}}$ is a semimartingale. By It\^o's formula for Dirichlet processes shown in \cite{Nakao}, it holds that
\begin{align*}
f((\phi u)(Z_t))-f((\phi u)(Z_0)) &= \sum_{j=1}^d\int_0^t D_j\bar{f}\circ (\phi u)(Z_{s-})\, dM^{[\phi u^j]}_s\\
&\h + \sum_{j=1}^d\int_0^t D_j\bar{f}\circ (\phi u)(Z_{s-})\, dN^{[\phi u^j]}_s\\
&\h + \text{(a process of bounded variation on each compact interval)}
\end{align*}
for all $t\in [0, \infty)$. Since
\[
D_j\bar{f}\circ (\phi u)(z)=\phi (z)(D_j\bar{f}\circ u)(z),\ z\in D_1,
\]
it holds that
\begin{align*}
\int_0^{t \wedge \tau_{D_1}}D_j\bar{f}\circ (\phi u)(Z_{s-})\, dN^{[\phi u^j]}_s= & \int_0^{t \wedge \tau_{D_1}}\phi \cdot (D_j\bar{f}\circ u)(Z_{s-})\, dN^{[\phi u^j]}_s
\end{align*}
for all $t\geq 0$, $\mathbb{P}_z$-a.s. q.e. $z$. By Fukushima decomposition, it holds that
\[
u^j(Z_t)-u^j(Z_0)=M^{[\phi u^j]}_t+N^{[\phi u^j]}_t,\ t<\tau_{D_1},\ \mathbb{P}_z\text{-a.s. q.e.}\ z\in E.
\]
Thus by \eqref{uhv}, we have
\begin{gather}
M^{[\phi u^j]}_t+N^{[\phi u^j]}_t=h^j(Z_t)-h^j(Z_0)+M^{[v^j]}_t+N^{[v^j]}_t,
\end{gather}
for $t<\tau_{D_1}$, $\mathbb{P}_z$-a.s. q.e. $z\in E$. Set $A^j_t$ and $B^j_t$ as in \lref{hAB} for the function $u^j$. Then $A^j$ is a process of $\mathbb{P}_z$-integrable variation for q.e. $z$ and $B^j$ is the dual predictable projection of $A$. Moreover it holds that
\begin{gather}
M^{[\phi u^j]}_t+N^{[\phi u^j]}_t + (A^j_t-B^j_t) + B^j_t=H^j_t+M^{[v^j]}_t+N^{[v^j]}_t,
\end{gather}
for $t\leq \tau_{D_1}$ and $j=1,\dots,d$, $\mathbb{P}_z$-a.s. q.e. $z\in E$. Since $H_t$ is an $\mathbb{R}^d$-valued $\mathbb{P}_z$-martingale for q.e. $z\in E$, we can decompose $H_t$ as
\[
H^j_t=H^{j,c}_t+H_t^{j,d},
\]
for each $j=1,\dots,d$ where $H^{j,c}$ is a $\mathbb{P}_z$-continuous martingale and $H^{j,d}$ is a $\mathbb{P}_z$-purely discontinuous martingale for q.e. $z\in E$. Thus $H^{j,d}_t + M^{[v^j],d}_{t\land \tau_{D_1}}-M^{[\phi u^j],d}_{t\land \tau_{D_1}}-(A^j_{t\land \tau_{D_1}}-B^j_{t\land \tau_{D_1}})$ is continuous and purely discontinuous local martingale and consequently it equals zero. Therefore it holds that
\[
N^{[\phi u^j-v^j]}=H_t^{j,c}+M^{[v^j],c}_t+M^{[\phi u^j],c}_t+B^j_t
\]
for $t\leq \tau_{D_1}$ and $j=1,\dots,d$, $\mathbb{P}_z$-a.s. q.e. $z\in E$. In the same way as in \rref{integralexittimerem}, we obtain
\[
\langle (H^{j,c}+M^{[v^j],c}+M^{[\phi u^j],c})^{\tau_{D_1}} \rangle_t=0,\ t\geq 0,\ \mathbb{P}_z\text{-a.s.}\ m\text{-a.e.}\ z\in E.
\]
Thus $N^{[\phi u^j-v^j]}_{t\land \tau_{D_1}}-B^j_{t\land \tau_{D_1}}$ is a CAF of $Z^{D_1}$ which is $m$-equivalent to zero and consequently, it equals $0$ for all $t\geq 0$, $\mathbb{P}_z$-a.s. q.e. $z \in D_1$. On the other hand it is obvious that $(H^{j,c}+M^{[v^j],c}+M^{[\phi u^j],c})^{\tau_{D_1}}=0$, $\mathbb{P}_z$-a.s. for q.e. $z\in E\, \backslash \, D_1$. Thus we obtain
\[
N^{[\phi u^j-v^j]}_t=B^j_t,\ t\leq \tau_{D_1},\ \mathbb{P}_z\text{-a.s. q.e.}\ z\in E
\]
and it holds that
\begin{align*}
\sum_{j=1}^d\int_0^{t\wedge \tau_{D_1}}\phi \cdot (D_jf\circ u)(Z_{s-})\, dN^{[\phi u^j]}_s=& \sum_{j=1}^d \int_0^{t\wedge \tau_{D_1}}\phi \cdot (D_jf\circ u)(Z_{s-})\, dN^{[v^j]}_s\\
&+\sum_{j=1}^d \int_0^{t\wedge \tau_{D_1}}\phi \cdot (D_jf\circ u)(Z_{s-})\, dB^{j}_s\\
=& \sum_{j=1}^d \int_0^{t\wedge \tau_{D_1}}\phi \cdot (D_jf\circ u)(Z_{s-})\, dB^{j}_s
\end{align*}
by \lref{integralPCAF}. Thus $(\phi f)(u(Z_t))^{\tau_{D_1}}$ is a semimartingale. Therefore we deduce that $u(Z)^{\tau_{D_1}}$ is a $\mathbb{P}_z$-semimartingale for q.e. $z\in E$ and so is $N^{[v^j],\tau_{D_1}}_t$ for each $j=1,\dots,d$. 
Thus in the same way as in \rref{integralexittime}, we can show that $N^{[v^j]}_{t\land \tau_{D_1}}$ is a CAF of locally bounded variation. This enables us to see the stochastic integral along $N^{[v^j]}_{t\land \tau_{D_1}}$ as the Stieltjes integral defined for $\mathbb{P}_z$-a.s. q.e. $z\in E$. Thus \eqref{uhv} and \eqref{intN} mean that $u$ is quasi-harmonic on $D$.
This completes the proof.
\end{proof}
\begin{proof}[Proof of \tref{harmonicmartingale2}]
Fix $\psi \in \mathcal{C}_D(u^*TM)$. Let $D_1$, $D_2$ be relatively compact open sets of $D$ such that $\mathrm{supp}[\psi]\subset D_1 \subset \overline{D_1} \subset D_2 \subset \overline{D_2} \subset D$ and $\phi$ a function satisfying $\phi \in C_0(E)\cap \mathcal{F}$, $0\leq \phi \leq 1$ and $\phi =1$ on $D_2$. Set $h$, $v$ as in the proof of \tref{harmonicmartingale1}. Then we can define the stochastic integral
\begin{align}
\int_0^t\psi^j (Z_{s-})\, du^j(Z_s)^{\tau_{D_1}}=& \int_0^t \psi^j (Z_{s-})\, dH^{j,\tau_{D_1}}_s+\int_0^t \psi^j(Z_{s-})\, dM^{[v^j],\tau_{D_1}}_s\nonumber \\
&+ \int_0^t \psi^j (Z_{s-})\, dN^{[v^j],\tau_{D_1}} \label{integraluvh}
\end{align}
for each $j=1,\dots,d$. Here the first two terms of the right-hand side are $\mathbb{P}_z$-local martingales. Moreover, since $u$ is quasi-harmonic, the left-hand side is a $\mathbb{P}_z$-local martingale for q.e. $z\in E$. Thus we can deduce that $\dis \sum_{j=1}^d\int_0^{t\land \tau_{D_1}}\psi^j (Z_{s-})\, N^{[v^j]}_s$ is a continuous local martingale additive functional for $Z^{D_1}$ and consequently it equals zero in the same way as in \rref{integralexittimerem}. Since $\phi \psi = \psi$, it holds that
\begin{align*}
\sum_{j=1}^d \mathcal{E}(\psi^j,v^j)&= \sum_{j=1}^d\mathcal{E}(\phi \psi^j,v^j)\\
&=- \lim_{t\to 0}\frac{1}{t} \mathbb{E}_{\phi m}\left[\sum_{j=1}^d \int_0^t\psi^j (Z_{s-})\, dN^{[v^j]} _s \right]\\
&= 0.
\end{align*}
Thus we obtain
\[
\sum_{j=1}^d \mathcal{E}(\psi^j, u^j)=\sum_{j=1}^d \mathcal{E}(\psi^j, h^j+v^j)=0.
\]
Thus $u$ is a weakly harmonic map on $D$.
\end{proof}
\begin{proof}[Proof of \tref{harmonicmartingale}]
Suppose that $M$ is compact and $u\in \mathcal{F}_{loc}^D(M)$. Then $u$ is bounded and satisfies \thetag{A} and \thetag{B} by Lemma 6.7.6 of \cite{ChenFuku}. Therefore the equivalence of the weak harmonicity and the quasi-harmonicity for $u$ holds by Theorems \ref{harmonicmartingale1} and \ref{harmonicmartingale2}.
\end{proof}
\begin{rem}
According to Theorem 2.9 of \cite{Chen}, the stronger result holds for harmonic functions than \tref{harmonicmartingale2} above. Let $D\subset E$ be an open set with $m(D)<\infty$ and $u$ a function on $E$ satisfying $u\in L^{\infty}(D;m)$ and condition $(A)$ so that $\{\tilde u(Z_{t\land \tau_D}\}_{t\geq 0}$ is a $P_z$-uniformly integrable martingale for q.e. $z\in E$. Then by Theorem 2.9 of \cite{Chen}, the function $u$ is automatically in $\mathcal{F}^D_{loc}$. We have not established such a result for harmonic maps as we have not taken into account the counterpart of the uniform integrability of martingales on manifolds. Without considering the condition analogous to the uniform integrability for martingales on manifolds, the counterpart of the statement of Theorem 2.9 of \cite{Chen} does not hold for harmonic maps with values in compact manifolds. Let $u$ be a map from the $2$-dimensional unit ball to the circle given as $u(z):=\frac{z}{|z|}$. Then $u$ is bounded and maps any Brownian motion to a martingale, but the energy of $u$ is infinite around the origin. 
\end{rem}
\section{Fine continuity of weakly harmonic maps}\label{pf2}
In this section, we consider the continuity of weakly harmonic maps along the paths of Markov processes in some situations. We assume that the target manifold $M$ is compact. First we refer to \cite{Oka23} for the definition of martingales on manifolds indexed by positive time.
\begin{dfn}
Let $\{X_t\}_{t>0}$ be a process indexed by $t\in (0,\infty)$. Suppose that there exist an $\{\mathcal{F}_t\}_{t\geq 0}$-stopping time $\zeta$ and a point $p\in \mathbb{R}^d$ such that $X_t \in M$ for $t\in [0,\zeta)$ and $X_t=p$ for $t\geq \zeta$. We call $(\{X_t\}_{t>0},\zeta,p)$ an $M$-valued martingale with an end point indexed by $t\in (0,\infty)$ if for all $\varepsilon>0$, $(\{X_{t+\ep}\}_{t\geq 0},(\zeta-\ep)\lor 0, p)$ is an $M$-valued $\{\mathcal{F}_{t+\ep} \}_{t\geq 0}$-martingale with the end point $p$.
\end{dfn}
\pref{harmonicmartingale3} below provides the equivalent condition for the continuity of weakly harmonic maps along the paths of Markov processes which is called the fine continuity. For an $M$-valued martingale $(\{X_t\}_{t>0},\zeta,p)$ with an end point, we denote the quadratic variation of $X$ as an $\mathbb{R}^d$-valued semimartingale by $[X,X]^{\ep}_t$ on the interval $(\varepsilon, t]$. We write the dual predictable projection of $[X,X]^{\ep}_t$ as $\langle X,X\rangle ^{\varepsilon}_t$.
\begin{prop}\label{harmonicmartingale3}
Let $D\subset E$ be an open set and $u\colon E \to M$ a Borel measurable map in $\mathcal{F}^D_{loc}(M)$ which is quasi-continuous on $D$ with respect to a regular Dirichlet form $(\mathcal{E},\mathcal{F})$. We suppose that $u$ is a weakly harmonic map on $D$. We further suppose that the transition function $\{p_t \}_{t\geq 0}$ of $Z$ satisfies the condition of absolute continuity, i.e. there exists a Borel measurable function $p\colon [0,\infty)\times E \times E \to [0,\infty]$ such that
\[
p_tf(z)=\int_Ef(w)p_t(t,z,w)\, m(dw),\ \text{for all}\ f\in \mathcal{B}_+(E).
\]
Then for all $z\in D$ and a relatively compact open set $D_1$ such that $\overline{D_1}\subset D$, the process $\{ u(Z_t)^{\tau_{D_1}}\}_{t>0}$ is an $M$-valued $(\mathbb{P}_z,\{\mathcal{F}_t \}_{t>0})$-martingale with the end point $0$ indexed by $t\in (0,\infty)$. Moreover, for each $z_0\in D_1$, the following are equivalent:
\begin{itemize}
\item[(i)]$\dis \lim_{t\to 0} u(Z_t)^{\tau_{D_1}}$ exists in $M$, $\mathbb{P}_z$-a.s.;
\item[(ii)]$\{ u(Z_t)^{\tau_{D_1}}\}_{t> 0}$ can be extended to an $M$-valued $\mathbb{P}_z$-martingale with an end point indexed by $t\geq 0$;
\item[(iii)]$\dis \lim_{\ep \to 0} [u(Z)^{\tau_{D_1}},u(Z)^{\tau_{D_1}}]^{\ep}_t<\infty$, $\mathbb{P}_z$-a.s.;
\item[(iv)]$\dis \lim_{\varepsilon \to 0}\langle u(Z)^{\tau_{D_1}}, u(Z)^{\tau_{D_1}} \rangle ^{\varepsilon}_t<\infty$, $\mathbb{P}_z$-a.s.
\end{itemize}
\end{prop}
\begin{proof}
By \tref{harmonicmartingale}, the map $u$ is quasi-harmonic on $D$. Moreover, by the absolute continuity of the transition function, the triple $(\{ u(Z_t)^{\tau_{D_1}} \}_{t>0}, \zeta, 0)$ is an $M$-valued $\mathbb{P}_z$-martingale indexed by $t\in (0,\infty)$ for all $z\in D_1$. (See Lemma 4.10 of \cite{Oka23} for details.) Thus the equivalence among (i) through (iii) in \pref{harmonicmartingale3} can be shown by Theorem 1.2 of \cite{Oka23}. The equivalence of (i) and (iv) can be verified by \lref{convtto02} below.
\end{proof}
\begin{lem}\label{convtto02}
Let $M$ be a compact submanifold of $\mathbb{R}^d$ and $(\{X_t\}_{t>0},\zeta,p)$ an $M$-valued martingale with an end point indexed by $t>0$. Then the following are equivalent:
\begin{itemize}
\item[(i)]$\displaystyle \lim_{t\to 0}X_t$ exists almost surely;
\item[(ii)]$\displaystyle \lim_{\varepsilon \to 0}\langle X,X\rangle ^{\varepsilon}_t<\infty$ almost surely.
\end{itemize}
\end{lem}
\begin{rem}
In (ii) of \lref{convtto02}, the existence of $\dis \lim_{\ep \to 0}\langle X,X \rangle^{\ep}_t$ is guaranteed by the monotonicity obtained from \eqref{nu} and \eqref{dualquad} below.
\end{rem}
\begin{proof}[Proof of \lref{convtto02}]
First we assume (i). Then we have $\displaystyle \lim_{\ep \to 0}[X,X]^{\ep}_t<\infty$ and consequently, the process $\{X_t\}_{t\geq 0}$ is an $M$-valued martingale with the end point $p$ by Theorem 1.2 of \cite{Oka23}. Let $A$ be the dual predictable projection of $[X,X]$. Then we have
\[
A_{t}-A_{\ep}=\langle X,X\rangle^{\ep}_{t}.
\]
Thus (ii) follows.

Next we assume (ii). We set $\displaystyle a=\sup_{0\leq t<\infty}|\Delta X_t|$. Then $a<\infty$ since $M$ is compact. We let the triple $(B^{\ep}, C^{\ep},\nu^{\ep})$ be the characteristics of the $d$-dimensional semimartingale $\{X_t\}_{t\geq \ep}$ with respect to the truncation function $\mathbf 1_{\{|x|<a \}}$, that is, $B_t$ be the predictable locally bounded variation part of the canonical decomposition of special semimartingale $\{X_t \}_{t\geq \ep}$,
\begin{gather*}
C^{\ep, i,j}_t:=[X^i,X^j]^{c,\ep}_t,
\end{gather*}
and $\nu^{\ep}$ be the good version of the dual predictable projection of the random measure
\[
\mu^{\ep}(\omega,(t_1,t_2]\times J)=\sum_{\ep<s\leq t}\mathbf{1}_{\{\Delta X_s(\omega)\neq 0 \}}\delta_{(s,\Delta X_s(\omega))}((t_1,t_2]\times J),
\]
for $0<t_1\leq t_2$, and $J \in \mathcal{B}(\mathbb{R}^d)$. We note that for $0<\ep_1<\ep_2$,
\begin{align}\label{Cij}
C^{\ep_1,i,j}_t=C^{\ep_2,i,j}_t+C^{\ep_1,i,j}_{t\land \ep_2}.
\end{align}
In particular, for $0<s<t$, the matrix
\[
\left\{ \left(C^{\ep_1,i,j}_t-C^{\ep_2,i,j}_t \right)-\left( C^{\ep_1,i,j}_s-C^{\ep_2,i,j}_s \right) \right\}_{i,j}
\]
is non-negative definite. Similarly, it holds that
\[
\mu^{\ep_1}(\omega, (t_1,t_2]\times J)=\mu^{\ep_2}(\omega, (t_1,t_2]\times J)+\mu^{\ep_1}(\omega, (t_1, t_2 \land \ep_2] \times J).
\]
Thus we have
\begin{align}\label{nu}
\nu^{\ep_1}(\omega, (t_1,t_2]\times J)=\nu^{\ep_2}(\omega, (t_1,t_2]\times J)+\nu^{\ep_1}(\omega, (t_1, t_2 \land \ep_2] \times J).
\end{align}
Let $f\in C^{\infty}(M)$ be any smooth function and $\bar{f}\in C^{\infty}_0(\mathbb{R}^d)$ the extension of $f$ given as \eqref{nabla}. By It\^o's formula, we have
\begin{align*}
\bar{f}(X_t)-\bar{f}(X_{\ep})&= \int_{\ep+}^t\langle D\bar{f}(X_{s-}),dX_s\rangle+\frac{1}{2}\int_{\ep+}^t\partial_i \partial_j \bar{f}(X_{s-})\, dC^{\ep,i,j}_s\\
&\h +\int_{\ep+}^t\int_{\mathbb{R}^d}\left( \bar{f}(X_{s-}+x)-\bar{f}(X_{s-}) \right)\, (\mu^{\ep}-\nu^{\ep})(dsdx)\\
&\h +\int_{\ep+}^t\int_{\mathbb{R}^d} W(s,x)\, \nu^{\ep}(dsdx)
\end{align*}
for $0<\ep \leq t$, where
\[
W(t,x):=\bar{f}(X_{s-}+x)-\bar{f}(X_{s-})-\langle D\bar{f}(X_{s-}),x\rangle.
\]
Moreover, for $0<\ep_1<\ep_2$, we have
\begin{align}
&\left| \int_{\ep_1+}^t\partial_i \partial_j \bar{f} (X_{s-})\, dC^{\ep_1,i,j}_s-\int_{\ep_2+}^t\partial_i \partial_j \bar{f} (X_{s-})\, dC^{\ep_2,i,j}_s\right| \nonumber \\
&\h \leq \| \mathrm{Hess} \bar{f}\| \sum_{i=1}^d \int_{\ep_1+}^t\, \left|d \left(C^{\ep_1,i,i}_t-C^{\ep_2,i,i}_t\right) \right| \nonumber \\
&\h \leq \| \mathrm{Hess} \bar{f}\| \left( [X^c,X^c]^{\ep_1}_t-[X^c,X^c]^{\ep_2}_t\right), \label{quadconti}\\
&\left| \int_{\ep_1+}^t\int_{\mathbb{R}^d}W(s,x)\,  \nu^{\ep_1}(dsdx)-\int_{\ep_2+}^t\int_{\mathbb{R}^d}W(s,x)\, \nu^{\ep_2}(dsdx)\right| \nonumber \\
&\h \leq \| \mathrm{Hess} \bar{f} \| \left( \int_{\ep_1+}^t\int_{\mathbb{R}^d}|x|^2\, \nu^{\ep_1}(dsdx)- \int_{\ep_2+}^t\int_{\mathbb{R}^d}|x|^2\, \nu^{\ep_2}(dsdx)\right)\label{quadjump}
\end{align}
by \eqref{Cij} and \eqref{nu}, where $\dis \|\mathrm{Hess} \bar{f} \| = \sup_{x\in \mathbb{R}^d}|\mathrm{Hess}\bar{f}(x)|$. Since it holds that
\begin{align}\label{dualquad}
\langle X,X\rangle^{\ep}_t= [X^c,X^c]^{\ep}_t+\int_{\ep+}^t\int_{\mathbb{R}^d}|x|^2\, \nu^{\ep}(dsdx),
\end{align}
the right-hand sides of \eqref{quadconti} and \eqref{quadjump} converge to $0$ as $\ep_2 \to 0$.
Thus we can define $U_t$, $V_t$ by
\begin{align*}
U_t&=\lim_{\ep \to 0} \int_{\ep+}^t \partial_i \partial_j \bar{f} (X_{s-})\, dC^{\ep,i,j}_s,\\
V_t&=\lim_{\ep \to 0} \int_{\ep+}^t\int_{\mathbb{R}^d}\left( \bar{f}(X_{s-}+x)-\bar{f}(X_{s-})-\langle D\bar{f}(X_{s-}),x\rangle  \right)\, \nu^{\ep}(dsdx).
\end{align*}
We set $H_t:=f(X_t)-U_t-V_t$. Then $\{ H_t\}_{t>0}$ is a local martingale with bounded jump since $\{X\}_{t>0}$ is an $M$-valued martingale with an end point. Thus in the same way as the proof of Theorem 1.2 of \cite{Oka23}, we can show that $\displaystyle \lim_{t \to 0}X_t$ exists $\mathbb{P}$-a.s.
\end{proof}
To consider the continuity of weakly harmonic maps in the topology with respect to the metric on $E$ and compare it with the fine continuity, hereafter we limit the scope of the metric space $E$ to either of the following two cases as domains of harmonic maps:
\begin{description}
\item[Case 1.]Let $m\geq 3$ and
\begin{align*}
\left\{
\begin{array}{ll}
\dis \mathcal{F}=H^1(\mathbb{R}^m) \\
\dis \mathcal{E}(u,v)=\int_{\mathbb{R}^m}\nabla u(z) \cdot \nabla v(z)\, dz.
\end{array}
\right.
\end{align*}
For a bounded smooth domain $D\subset \mathbb{R}^m$ with smooth boundary and $u\in H^1_{loc}(D)$, we set
\begin{gather*}
\what{\mathcal{E}}^D(u):=\int_D|\nabla u|^2(z)\, dz.
\end{gather*}
In this case, if we set
\[
\what{H}^1(D):=\{ u\in L^2_{loc}(\mathbb{R}^m) \mid \widehat{\mathcal{E}}^D(u)<\infty \},
\]
then $\what{H}^1(D)=H^1(D)$.
\item[Case 2.]Let $m\geq 2$, $\alpha \in (0,2)$ and
\begin{align*}
\left\{
\begin{array}{ll}
\dis \mathcal{F}=H^{\frac{\alpha}{2}}(\mathbb{R}^m) \\
\dis \mathcal{E}(u,v)=c_{m, \alpha}\int_{\mathbb{R}^m\times \mathbb{R}^m}\frac{(\tilde u(z)-\tilde u(w))(\tilde v(z)-\tilde v(w))}{|z-w|^{m+\alpha}} \, dwdz,
\end{array}
\right.
\end{align*}
where
\[
c_{m,\alpha}=\alpha 2^{\alpha-2}\pi^{-\frac{m+2}{2}}\sin \left( \frac{\alpha \pi}{2}\right) \Gamma \left( \frac{m+\alpha}{2}\right)  \Gamma \left( \frac{\alpha}{2}\right).
\]
For a bounded smooth domain $D\in \mathbb{R}^m$ with smooth boundary and $u\in L^2_{loc}(\mathbb{R}^m)$, we set
\begin{align*}
\widehat{\mathcal{E}}^D(u)&:=c_{m,\alpha}\int \int_{D\times D}\frac{|u(z)-u(w)|^2}{|z-w|^{m+\alpha}}\, dwdz\\
&\h +2c_{m,\alpha}\int \int_{D\times (\mathbb{R}^m\, \backslash \, D)}\frac{|u(z)-u(w)|^2}{|z-w|^{m+\alpha}}\, dwdz,\\
\widehat{H}^{\frac{\alpha}{2}}(D)&=\{ u\in L^2_{loc}(\mathbb{R}^m) \mid \widehat{\mathcal{E}}^D(u)<\infty \}.
\end{align*}
See \cite{MPS} for details about the space $\widehat{H}^{\frac{\alpha}{2}}(D)$.
\end{description}
The corresponding Markov processes $Z$ are Brownian motion and $\alpha$-stable process in Case 1 and Case 2, respectively.
\begin{rem}\label{counter}
In \cite{Riv95}, T. Rivi\'ere constructed a weakly harmonic map $u$ with respect to the Dirichlet energy from the $3$-dimensional unit ball $B_1$ into the $2$-dimensional unit sphere $\mathbb{S}^2$ which is everywhere discontinuous. On the other hand, $u$ has a quasi-continuous modification since the Dirichlet energy of $u$ is finite. Thus there is a difference between the continuity in the Euclidean topology and that in the fine topology if we do not impose any restrictions for harmonic maps.
\end{rem}
\begin{dfn}\label{tangentmap}
We assume the setting of either Case 1 or Case 2. Let $u\in \mathcal{F}^D_{loc}(M)$ be a weakly harmonic map. Fix $z_0 \in D$ and $\rho>0$ such that $\overline{B_{\rho}(z_0)}\subset D$. Then $\varphi \colon \mathbb{R}^d\to M$ is called a tangent map of $u$ at $z_0$ if for all $R>0$, there exists a decreasing sequence $\{r_i \}_{i=1}^{\infty}$ with $r_i\to 0$, $r_1<\frac{\rho}{R}$ such that
\[
u_{z_0,r_i} \to \varphi \ \text{weakly in}\ (\what{\mathcal{F}}^{B_R},\what{\mathcal{E}}^{B_R})\ \text{as}\ i \to \infty,\ \text{where}\ u_{z_0,r}(z):=u(z_0+rz),
\]
where
\[
\what{\mathcal{F}}^{B_R}:=\{v \colon E \to \mathbb{R}^d \mid v^i \in L^2_{loc}(\mathbb{R}^d),\ i=1,\dots,d,\ \what{\mathcal{E}}^{B_R}(v)<\infty \}.
\]
\end{dfn}
To see the equivalence of the continuity in Euclidean topology and that in fine topology in some cases, we impose some conditions on weakly harmonic maps.
\begin{as}\label{astangent}
We assume the setting of either Case 1 or Case 2. Let $u\in \mathcal{F}^D_{loc}(M)$ be a weakly harmonic map with $\what{\mathcal{E}}^D(u)$. For each $z_0\in D$ and tangent map $\varphi$ of $u$ at $z_0$, we assume that $u$ satisfies the following two conditions.
\begin{description}
\item[(i)]Under the notation in \dref{tangentmap}, there exists a subsequence $\{r_{i_j} \}$ such that
\[
\what{\mathcal{E}}^{B_R} (u_{z_0,r_{i_j}}- \varphi) \to 0\ \text{as}\ j \to \infty.
\]
\item[(ii)]$u$ and $\varphi$ have modifications which are continuous quasi-everywhere. In addition, $u$ is continuous at $z_0$ if and only if any tangent map at $z_0$ is constant almost everywhere.
\end{description}
\end{as}
We remark that there are enough examples satisfying \asref{astangent} as follows.
\begin{ex}\label{ex1}
In Case 1, we set
\[
H^1(D;M):=\{u \in \mathcal{F}^D_{loc}(M) \mid \what{\mathcal{E}}^D(u) < \infty \}.
\]
We say that $u\in H^1(D;M)$ is an energy minimizing map on $D$ if it satisfies
\[
\what{\mathcal{E}}^D(u)\leq \what{\mathcal{E}}^D(u')
\]
for all $u'\in H^1(D;M)$ with
\[
\mathrm{supp}[u-u']\subset D,
\]
where $\mathrm[\cdot]$ is the support of a vector-valued function. Every energy minimizing map $u$ is a weakly harmonic map and \asref{astangent} is satisfied by the result of \cite{SU82}.
\end{ex}

\begin{ex}\label{ex2}
In Case 1, we further suppose $M=\mathbb{S}^n$. A weakly harmonic map $u\in H^1(D; \mathbb{S}^n)$ is said to be stationary if 
\[
\left. \frac{d}{dt} \right|_{t=0}\what{\mathcal{E}}(u\circ \Phi_t)=0\ \text{for all}\ \mathcal{X}\in C_0^1(D;\mathbb{R}^m),
\]
where $\{\Phi_t\}_{t\in \mathbb{R}}$ is the integral curves of the vector field $\mathcal{X}$. Furthermore, $u$ is said to be stable if
\[
\left. \frac{d^2}{dt^2} \right|_{t=0}\what{\mathcal{E}}\left( \frac{u + t\psi}{|u+t\psi|} \right)\geq 0
\]
for all $\psi \in C^{\infty}_0(D;\mathbb{R}^{n+1})$ with $u(z)\cdot \psi (z)=0$ a.e. $z \in D$. If $u\in H^1(D;\mathbb{S}^n)$ is a stable stationary harmonic map, then \asref{astangent} is satisfied by \cite{Evans, HW99}.
\end{ex}

\begin{ex}\label{ex3}
In Case 2, we further suppose that $M=\mathbb{S}^n$. Then we can define minimizing or stationary harmonic maps with respect to the fractional Dirichlet energy in the same way as \eref{ex1} and \eref{ex2}. In \cite{MPS}, it was shown that \asref{astangent} is satisfied in the following two cases.
\begin{itemize}
\item $\alpha \in (0,2)\, \backslash \, \{1\}$, $n>\alpha$, and $u$ is a stationary $\frac{\alpha}{2}$-harmonic map.
\item $u$ is a minimizing $\frac{1}{2}$-harmonic map.
\end{itemize}
\end{ex}
Before stating \pref{minimizing} below, we recall a censored $\alpha$-stable process and a reflecting $\alpha$-stable process.
For an open set $D$, we set
\begin{align}\label{reflecting}
\left\{
\begin{array}{ll}
\dis \mathcal{C} = \{ u\in L^2(D)\cap \mathcal{B}(D) \mid \mathcal{D}(u,u)<\infty \}, \\
\dis \mathcal{D} (u,v)=c_{m, \alpha}\int_{D \times D}\frac{(u(z)-u(w))(v(z)-v(w))}{|z-w|^{m+\alpha}} \, dwdz,
\end{array}
\right.
\end{align}
Then $(\mathcal{D}, \mathcal{C})$ is a Dirichlet form on $L^2(D)$. 
If we further suppose that $D$ is a smooth domain, then the space $\mathcal{C}$ is identical to $H^{\frac{\alpha}{2}}(D)$ and the pair $(\mathcal{D}, \mathcal{C})$ becomes a regular Dirichlet form on $L^2(\overline{D})$. The corresponding Hunt process $(Z^*, \mathbb{P}_z^*)$ is called a reflecting $\alpha$-stable process. 

For a process $H$ and a stopping time $\tau$, we set
\[
H^{\tau-}_t:=H_t\mathbf{1}_{[0,\tau)}(t)+H_{\tau-}\mathbf{1}_{[\tau,\infty)}
\]

\begin{prop}\label{minimizing}
In addition to the setting of \pref{harmonicmartingale3}, we further assume the setting of either Case 1 or Case 2. Let $u\in \mathcal{F}^D_{loc}(M)$ be a weakly harmonic map satisfying \asref{astangent}. Then for $z_0\in D$, the following (v) is also equivalent to (i)--(iv) in \pref{harmonicmartingale3}.
\begin{itemize}
\item[(v)]$u$ has a modification which is continuous at $z_0$.
\end{itemize}
\end{prop}
\begin{rem}\label{thinness}
As for Case 1, we can show \pref{minimizing} above by the application of Wiener's criterion from classical potential theory. If $u$ is finely continuous at $z_0$, then it holds that
\begin{align}\label{line}
\lim_{r\to 0} u(z_0+rz)= u(z_0)\ \text{for a.e.}\ z\in D
\end{align}
by \cite{Deny48} or Theorem 7.8.1 of \cite{AG}. Let $\varphi$ be a tangent map at $z_0$. Fix $\rho>0$ such that $\overline{B_{\rho}(z_0)}\subset D.$ Then (i) of \asref{astangent} enables us to take a decreasing sequence $r_1\geq r_2 \geq \dots \to 0$ such that for all $R>0$,
\[
\what{\mathcal{E}}^{B_{R}} (u_{z_0,r_{i}}- \varphi) \to 0\ \text{as}\ i \to \infty.
\]
By the Sobolev embedding theorem, we have
\[
\lim_{i\to \infty}\int_{B_{R}}|u_{z_0,r_i}(z)-\varphi(z) |^2\, dz=0.
\]
Thus by the bounded convergence theorem, we have
\[
\int_{B_{R}}|u(z_0)-\varphi(z)|^2\, dz=0.
\]
This means that any tangent map of $u$ at $z_0$ is a.e. constant. Thus by (ii) of \asref{astangent}, $u$ is continuous at $z_0$.

Unfortunately, the same proof does not work in Case 2. Indeed, we can construct a function $f \colon \mathbb{R}^m \to \mathbb{R}$ which is finely continuous at $0$ but does not have a modification satisfying \eqref{line}. Let $m\geq 2$ and $\alpha \in (0,1]$. We set
\[
S_r^{\ep}:=\{ z\in \mathbb{R}^m \mid r-\ep \leq |z| \leq r+\ep\}
\]
for $0< \ep<r$. Since $\alpha \in (0,1]$ yields $\mathrm{Cap}^0(S_r)$=0, we have
\[
\lim_{\ep \to 0}\mathrm{Cap}^0(S_r^{\ep})=0,
\]
where $\mathrm{Cap}^0$ is the outer $0$-capacity with respect to the Riesz kernel $g$ with index $\alpha$ defined by
\[
g(z,w)=\frac{c_{m,\alpha}}{|z-w|^{m-\alpha}}.
\]
For fixed $\lambda >1$, we set
\[
r_n:= \frac{1}{2}c_{m,\alpha}^{\frac{1}{m-\alpha}} \left( \lambda^{-\frac{n}{m-\alpha}}-\lambda^{-\frac{n+1}{m-\alpha}} \right)
\]
and take $\ep_n>0$ such that
\[
\mathrm{Cap}^0(S_{r_n}^{\ep_n})<2^{-n} \lambda^{-n}.
\]
We set
\begin{align*}
S&:= \bigcup_{n=1}^{\infty} S_{r_n}^{\ep_n},\\
A_n&:=\{ z\in S \mid \lambda ^n \leq g(0,z) \leq \lambda^{n+1} \}.
\end{align*}
Then
\[
\sum_{n=1}^{\infty}\lambda^n\mathrm{Cap}^0(A_n \cap S)<1.
\]
Thus the Wiener's criterion for the Riesz kernel yields that $S$ is thin at $0$, i.e.
\[
\mathbb{P}_0\left(\sigma_S=0 \right)=0.
\]
(We can completely track the probabilistic proof of Wiener's criterion in \cite{Bass} even in the case of the Riesz kernel.) Therefore, if we set $f:= \mathbf{1}_S$, then $f$ is finely continuous at $0$ but obviously there does not exist a modification of $f$ satisfying \eqref{line}.
\end{rem}
\begin{proof}[Proof of \pref{minimizing} in Case 2]
We only have to show that (i) yields (v). Fix $R>0$. Then we can assume that the sequence $\{r_i \}_{i=1}^{\infty}$ satisfies $B_{r_1R}(z_0) \subset D$ without loss of generality. We set $\varphi_j := u_{z_0,r_{i(j)}}-\varphi$. Let $(\mathcal{D}, \mathcal{C})$ be a Dirichlet form defined in the same way as \eqref{reflecting} on $B_R$ and $(Z^*, \mathbb{P}_z^*)$ be a reflecting $\alpha$-stable process on $B_R$. Then by assumption,
\begin{align*}
\mathcal{D}(\varphi_j,\varphi_j)\leq \widehat{\mathcal{E}}^{B_R}(\varphi_j,\varphi_j) \to 0\ \text{as}\ j \to \infty.
\end{align*}
This implies that
\begin{align*}
\mathbb{P}^*_z\left( \varphi_j (Z^*_t) \to 0\ \text{uniformly in}\ t\in [0,\infty)\ \text{as}\ j \to \infty \right)=1\ \text{for q.e.}\ z\in B_R.
\end{align*}
Since $Z^*$ has the transition density, we have
\begin{align*}
\mathbb{P}^*_z\left( \varphi_j (Z^*_t) \to 0\ \text{uniformly on any compact subset of}\ (0, \infty)\ \text{as}\ j \to \infty \right)=1\ \text{for all}\ z\in B_R.
\end{align*}
Since $Z^*$ has the same law as the part process of $Z$ on $[0, \tau_{B_R})$, we have
\begin{align*}
\mathbb{P}_z\left( \varphi_j (Z_{t \land \tau_{B_R}})^{\tau_{B_R}-} \to 0\ \text{uniformly on any compact subset of}\ (0, \infty)\ \text{as}\ j \to \infty \right)=1
\end{align*}
for any $z\in D$ and this holds for any fixed $R>0$. On the other hand, for any bounded continuous function $F$ on $\mathbb{R}^d$,
\begin{align*}
\mathbb{E}_{z_0}\left[ F(u_{z_0,r_i} (Z_t)^{\tau_{B_R}-} ) \right] &= \mathbb{E}_0\left[ F(u(z_0+Z_{r^{\alpha}_i t})^{\tau_{\alpha,r_i}-} ) \right]\\
& \to F(u(z_0))
\end{align*}
by the scale invariance and bounded convergence theorem, where
\begin{align*}
\tau_{\alpha,r}=\inf \{s\geq 0\mid Z_s \notin B_{r^{\frac{\alpha}{2}}R}\}.
\end{align*}
This yields that the law of $\varphi(Z_t)^{\tau_{B_R}-}$ concentrates at $u(z_0)$ for each $t>0$. This implies that $\varphi(Z)^{\tau_{B_R}}$ is constant along the path of $Z^{B_R}$ almost surely under $\mathbb{P}_0$ since it is a \cl process. Consequently, $\varphi(Z)$ is constant $\mathbb{P}_0$-a.s. since we took an arbitrary $R>0$. Take any continuous point $z\in \mathbb{R}^d$ of $\varphi$. Then for any $\eta >0$, there exists $\delta>0$ such that for all $w\in B_{\delta}(z)$, $|\varphi(z)-\varphi(w)|<\eta$. On the other hand, since $Z$ is irreducible, $\mathbb{P}_0(\{ \omega ; Z(\omega)\cap B_{\delta}(z)\neq \emptyset \})>0$. Thus by taking $t>0$ and $\omega$ such that the path $t \mapsto \varphi(Z_t(\omega))$ is constant and $Z_t(\omega)\in B_{\delta}(z)$, we have
\[
|\varphi(z)-\varphi(0)|=|\varphi(z)-\varphi(Z_t(\omega))|<\eta.
\]
Therefore $\varphi = \varphi(0)$ a.e. This implies $u$ is continuous at $z_0$ by \asref{astangent}.
\end{proof}
\begin{ex}
In Case 2, a weakly harmonic map $u\in \mathcal{F}^D_{loc}(M)$ with $\what{\mathcal{E}}(u)<\infty$ satisfying \asref{astangent} is continuous at $z_0\in D$ if and only if
\[
\lim_{\ep \to 0}\int_{\ep \land \tau_{D_1}}^{t\land \tau_{D_1}}\int_{\mathbb{R}^m}\frac{|u(Z_s)-u(z)|^2}{|Z_s-z|^{m+\alpha}}\, dzds<\infty,\ \Prob_{z_0}\text{-a.s.}
\]
for a relatively compact open subset $D_1$ in $D$ with $z_0\in D_1$.
\end{ex}
\begin{rem}
At the moment we do not know if the same argument holds or not for all stationary harmonic maps since our proof of \pref{minimizing} relies on \asref{astangent}.
\end{rem}

\section*{Acknowledgements}
The author would like to express his gratitude to Professor Hariya, his supervisor, for his careful reading of the manuscript and for helpful comments. The author also thanks a referee for his/her constructive comments and corrections which have led to significant improvement of this paper.

\section*{Statements and Declarations}
{\bf Competing interests:} No conflict is related to this article, and the author has no relevant financial or non-financial interests to disclose.

{\bf Ethical Approval:} Not applicable to this article.

{\bf Funding:} This work was supported by JSPS KAKENHI Grant Number 22J11051.

{\bf Data Availability Statements:} Data sharing is not applicable to this article as no datasets were generated or analyzed during the current study.

\end{document}